\documentclass[a4paper,12pt, reqno]{amsart}
\usepackage{amsfonts}
\usepackage{amsmath, array, bigints}
\usepackage{amssymb}
\usepackage{mathrsfs}
\usepackage[colorlinks]{hyperref}
 \usepackage{graphicx}
\usepackage{enumerate}

 \usepackage{esint}
\usepackage[usenames]{color}

\usepackage{mathtools}
\makeatletter
\newcommand{\setword}[2]{%
  \phantomsection
  #1\def\@currentlabel{\unexpanded{#1}}\label{#2}%
}
\makeatother
\newtheorem{thm}{Theorem}[section]
\newtheorem{cor}[thm]{Corollary}
\newtheorem{lem}[thm]{Lemma}

\numberwithin{equation}{section}
\usepackage[english]{babel} 
\usepackage{blindtext}

\theoremstyle{definition}
\newtheorem{definition}[thm]{Definition}

\begin{document}

\allowdisplaybreaks 
\def\dist{{\operatorname{dist}}}
\def\loc{{\operatorname{loc}}}
\def\ess{{\operatorname{ess~sup}}}
\def\essi{{\operatorname{ess~inf}}}
\def\esslim{{\operatorname{ess~lim~inf}}}
\def\supp{{\operatorname{supp}}}
\renewcommand{\d}{\:\! \mathrm{d}}


 \title[Harnack inequality of superposition operators ]{Harnack inequality for superposition operators of mixed fractional order}

 \author[Souvik Bhowmick, Sekhar Ghosh, Vishvesh Kumar, and R. Lakshmi]{Souvik Bhowmick, Sekhar Ghosh, Vishvesh Kumar and R. Lakshmi}

\address[Souvik Bhowmick]{Department of Mathematics, National Institute of Technology Calicut, Kozhikode, Kerala, India - 673601}
\email{souvikbhowmick2912@gmail.com / souvik\_p230197ma@nitc.ac.in}

\address[Sekhar Ghosh]{Department of Mathematics, National Institute of Technology Calicut, Kozhikode, Kerala, India - 673601}
\email{sekharghosh1234@gmail.com / sekharghosh@nitc.ac.in}
\address[Vishvesh Kumar]{Department of Mathematical Sciences, Indian Institute of Technology (BHU),
Varanasi, Uttar Pradesh, 221005, India.}
\email{vishveshmishra@gmail.com / vishvesh.mat@iitbhu.ac.in}
\address[R. Lakshmi]{Department of Mathematics, National Institute of Technology Calicut, Kozhikode, Kerala, India - 673601}
\email{lakshmir1248@gmail.com / lakshmi\_p220223ma@nitc.ac.in}
\date{}

\begin{abstract} The main aim of this paper is to establish the H\"older continuity and the Harnack inequality for weak solutions to Dirichlet problems associated with superposition operators of mixed fractional order, thereby complementing our previous work \cite{BGKL2026}.

To achieve this, we extend the De Giorgi--Nash--Moser theory to the framework of superposition operators by introducing a novel {\it nonlocal superposition tail}, which appears to be the first contribution of its kind in the literature. The obtained results are new even in the classical linear case $p=2$, thereby illustrating the broader applicability of the analytical techniques developed in this work.

As intermediate steps toward the proof of the main results, we also establish a logarithmic estimate for weak supersolutions, local boundedness for weak subsolutions, a weak Harnack inequality for weak supersolutions, an expansion of positivity for weak supersolutions, and tail estimates for weak solutions.

\end{abstract}


\maketitle

\section{Introduction and main results}
In this paper, we investigate the regularity properties of weak solutions associated with the nonlinear fractional superposition operator
\begin{equation}\label{SP}
    A_{\mu,p}u:=\int_{[0,1]}(-\Delta)_p^s u\,\d\mu(s)
    =\int_{(0,1)}(-\Delta)_p^s u\,\d\mu(s)-\alpha\Delta_pu,
\end{equation}
where $p\in(1,\infty)$, $\mu$ is a nonnegative, nontrivial finite Borel measure on $[0,1]$, $\alpha:=\mu(\{1\})\geq 0$, and $\mu(\{0\})=0$. The operator $\Delta_p$ denotes the local $p$-Laplacian,
$$\Delta_pu=\operatorname{div}\bigl(|\nabla u|^{p-2}\nabla u\bigr)$$
while $(-\Delta)_p^s$ stands for the fractional $p$-Laplacian, defined for $s\in(0,1)$ by
$$(-\Delta)_p^su(x)
:=2C_{N,p,s}\lim_{\varepsilon\searrow 0}
\int_{\mathbb{R}^N\setminus B_\varepsilon(x)}
\frac{|u(x)-u(y)|^{p-2}(u(x)-u(y))}
{|x-y|^{N+sp}}\,\d y.$$

The normalization constant $C_{N,p,s}>0$ is chosen so that the family $\{(-\Delta)_p^s\}_{s\in[0,1]}$ exhibits the consistent limiting behaviors
$$\lim_{s\searrow 0}(-\Delta)_p^su
=(-\Delta)_p^0u
:=|u|^{p-2}u,$$
and
$$\lim_{s\nearrow 1}(-\Delta)_p^su
=(-\Delta)_p^1u
:=-\Delta_pu
=-\operatorname{div}\bigl(|\nabla u|^{p-2}\nabla u\bigr).$$

Therefore, the operator $A_{\mu,p}$ provides a unified framework that interpolates between local and nonlocal nonlinear diffusion operators through the measure $\mu$. Our primary objective is to establish a Harnack inequality and an improved H\"older regularity for weak solutions associated with this class of mixed superposition operators.

The general superposition operator $B_{\upsilon,p}$ of nonlinear fractional operators was introduced by Dipierro \textit{et al.} in \cite{DPSV2} for the case $p=2$ and subsequently extended to the range $1<p<\infty$ in \cite{DPSV}. It is defined by
\begin{equation}\label{superposition operator}
B_{\upsilon,p}u:=\int_{[0,1]}(-\Delta)_p^s u\,\d\upsilon(s),
\end{equation}
where $\upsilon$ is a signed measure on $[0,1]$ admitting the Jordan decomposition
$$\upsilon=\upsilon^+-\upsilon^-,$$
with $\upsilon^+$ and $\upsilon^-$ being nonnegative finite Borel measures satisfying suitable assumptions. The operator $B_{\upsilon,p}$ provides a unified framework encompassing a broad class of local and nonlocal nonlinear diffusion operators. In particular, the operator $A_{\mu,p}$ considered in this paper arises as a special case of $B_{\upsilon,p}$ when the underlying measure is nonnegative.

Before presenting our main results, we briefly review some related developments in the literature. Besides their intrinsic mathematical interest, superposition operators arise naturally in various applications, including mathematical biology and population dynamics, where they are used to model different dispersal mechanisms, ranging from standard Gaussian diffusion to L\'evy flight processes; see, for instance, \cite{DPSV25,DV21}. 

Several existence and qualitative results for superposition operators have been established in recent years; we refer to Dipierro \textit{et al.} \cite{DPSV25,DPSV2,DPSV}, Afonso \textit{et al.} \cite{ABB2025}, Bisci \textit{et al.} \cite{BMS2025}, Proietti Lippi and Sportelli \cite{PS2026}, the middle two authors together with Aikyn and Ruzhansky \cite{AGKR2025}, and the first three authors \cite{BGK2026a, BGK2026}.

Concerning maximum principles, Dipierro \textit{et al.} \cite{DLSV2025} showed, in the case $p=2$, that the maximum principle generally fails for the operator \eqref{superposition operator} due to the possible sign-changing nature of the measure $\upsilon$. Nevertheless, they proved that the maximum principle remains valid for the superposition operator \eqref{SP}. More recently, for $p\in(1,\infty)$, the middle two authors, together with Aikyn and Ruzhansky \cite{AGKR2025i}, established weak and strong maximum principles as well as logarithmic estimates for \eqref{SP} by introducing the following nonlocal superposition tail:
\begin{equation}\label{tail1}
        \operatorname{{Tail_1}}(u;x_0,r)=\left[\int_{(0,1)}r^{sp}\left(\int_{\mathbb{R}^N\setminus B_r(x_0)}\frac{|u(y)|^{p-1}}{|y-x_0|^{N+sp}}dy\right)d \mu(s)\right]^{\frac{1}{p-1}}.
    \end{equation}
    where $u \in W^{s,p}_{\mu}(\mathbb{R}^N)$ (see Section \ref{pre}) in a open ball $B_r(x_0)\subset \mathbb{R}^N$. {\it These results are obtained for \eqref{SP}, as the maximum principle is valid if and only if $\upsilon$ maintains a constant sign.} More recently, in the case $p=2$, Dipierro \textit{et al.} \cite{DLSV2026} established higher regularity for weak solutions of the mixed local--nonlocal superposition problem, proving that if $u\in H_0^1(\Omega)$ is a weak solution, then $u\in W^{2,q}(\Omega)$. Specifically, they considered the problem
\begin{align*}
    \int_{[0,\Tilde{s}]}(-\Delta)^{s} u \d \upsilon(s)- \Delta u=&f(x) \text{ in  }\Omega, \nonumber \\
    u=&0 \text{ in  }\mathbb{R}^N\setminus \Omega,
\end{align*}
where $\Omega$ is open bounded subset of $\mathbb{R}^N$ with boundary class of $C^1$, $\Tilde{s}\in [0,\frac{1}{2})$, $f\in L^q(\Omega)$ with $q\in (1,\infty)$. 

We now discuss several regularity results corresponding to specific examples of superposition operators. Indeed, a comprehensive regularity theory for the general operator \eqref{MP} is not yet available. In the particular case where $\mu=\delta_1$, the superposition operator coincides with the $p$-Laplacian, namely $A_{\mu,p}=-\Delta_p$, for which a rich regularity theory has been developed; see the classical references \cite{Lind19,MZ97}. It is worth recalling that, in the local framework, the Harnack inequality and H\"older continuity are known to be equivalent for a large class of elliptic problems.

When $\mu$ is the Dirac measure concentrated at $s\in(0,1)$, the superposition operator reduces to the fractional $p$-Laplacian, that is, $A_{\mu,p}=(-\Delta)^s_p$. In this setting, a Harnack inequality is known to hold for globally nonnegative solutions. In the linear case $p=2$, Kassmann \cite{Kass11} showed that the classical Harnack inequality fails for sign-changing (nodal) solutions and provided a counterexample demonstrating that the global nonnegativity assumption is indispensable. To overcome this limitation, he introduced a modified Harnack inequality containing an additional term on the right-hand side, namely the {\it tail}, which quantifies the influence of the solution outside the domain and reflects the intrinsically nonlocal nature of the fractional Laplacian.

Subsequently, Di Castro, Kuusi, and Palatucci \cite{DKP2014,DKP2016} developed a comprehensive regularity theory for the fractional $p$-Laplacian with $p\in(1,\infty)$. By systematically exploiting the nonlocal tail, they established several fundamental results, including Caccioppoli-type inequalities, local boundedness, logarithmic estimates, local H\"older continuity, Harnack inequalities, and weak Harnack inequalities. Since then, the regularity theory of nonlocal equations has become an active and rapidly growing research area, generating a vast literature that cannot be surveyed exhaustively here. We refer the reader to \cite{Cozz17,DKP2014,DKP2016,Kass07} and the references therein for further developments.

In the mixed local--nonlocal setting corresponding to $\mu=\delta_1+\delta_s$ with $s\in(0,1)$, the superposition operator takes the form
$A_{\mu,p}=(-\Delta)_p+(-\Delta)^s_p$. In the linear case $p=2$, Foondun \cite{Foo09} established the Harnack inequality and local H\"older continuity for nonnegative solutions. Later, Garain and Kinnunen \cite{GK2022} developed a nonlinear regularity theory for $p\in(1,\infty)$ by introducing a suitable nonlocal tail adapted to the mixed local--nonlocal structure. Their work yielded several fundamental estimates, including the Caccioppoli inequality, local boundedness, logarithmic estimates, local H\"older continuity, the Harnack inequality, and the weak Harnack inequality. We also refer the reader to \cite{DeFG24,Foo09,GK2022} and references therein.

Very recently, inspired by the pioneering work of Di Castro, Kuusi, and Palatucci \cite{DKP2016}, we initiated the regularity theory for general superposition operators $A_{\mu,p}$. In particular, we established in \cite{BGKL2026}, for the first time in this general framework, Caccioppoli-type estimates, local boundedness, and weak Harnack inequalities through the introduction of the following normalized tail, which naturally encodes the contribution of all fractional orders present in the superposition operator:
\begin{equation}\label{tail2}
        \operatorname{{Tail_2}}(u;x_0,r)=\left[\int_{(0,1)}C_{N,s,p}r^{sp}\left(\int_{\mathbb{R}^N\setminus B_r(x_0)}\frac{|u(y)|^{p-1}}{|y-x_0|^{N+sp}}dy\right)d \mu(s)\right]^{\frac{1}{p-1}}.
    \end{equation}
Moreover, in \cite{BGKL2026}, we established local H\"older continuity for weak solutions of $A_{\mu,p}$ under the additional assumption that $\bar{s}=\inf_{\supp\{\mu\}}s>0$. This assumption plays a crucial role in our approach, since the argument relies heavily on the behavior of the normalized tail \eqref{tail2}. In particular, the method breaks down when $\bar{s}=0$, where both the superposition operator $A_{\mu,p}$ and the tail \eqref{tail2} exhibit significantly different behavior. Furthermore, establishing a Harnack inequality for weak solutions for Dirichlet problems involving $A_{\mu,p}$  presents additional difficulties arising from the interaction between the superposition structure of $A_{\mu,p}$  and the nonlocal tail term \eqref{tail2}, even in the case $\bar{s}>0$.

The primary goal of this paper is to establish local H\"older continuity and Harnack inequalities for weak solutions of Dirichlet problems associated with the superposition operator $A_{\mu,p}$ without imposing any additional assumptions on the measure $\mu$. In particular, we remove the restriction $\bar{s}>0$, which was essential in our previous work \cite{BGKL2026}, and develop a new approach capable of handling the singular regime where the support of $\mu$ may accumulate at zero.

Let us now present the assumptions on the superposition operator and discuss their applicability to a broad class of problems. We consider the measure $\mu$ associated with the superposition operator
$$
A_{\mu, p} u:=\int_{[0,1]}(-\Delta)_{p}^{s} u \d \mu(s)
=\int_{(0,1)}(-\Delta)_{p}^{s} u \d \mu(s)-\alpha \Delta_p u,
$$
where $\mu$ is a nonnegative finite Borel measure on $[0,1]$, nontrivial on $(0,1)$, and satisfies $\mu(\{0\})=0$. The assumption $\mu(\{0\})=0$ is imposed to ensure that the operator $A_{\mu,p}$ preserves the appropriate scaling properties. We set
$$
\alpha := \mu(\{1\}) \ge 0, \qquad \Sigma := \operatorname{supp}(\mu).
$$

Since $\mu$ is nontrivial, the support $\Sigma$ is nonempty. Consequently, there exists an exponent $s'\in (0,1]$ such that
\begin{equation}\label{m4}
\mu\left([s',1]\right)>0.
\end{equation}
Furthermore, by \eqref{m4}, there exists an exponent $\bar{s_\sharp}\in [s',1]$ satisfying $\mu\left([\bar{s_\sharp},1]\right)>0.$ Accordingly, we define
\begin{equation}
    s_\sharp =
    \begin{cases}
        \bar{s_\sharp}, & \text{if } \alpha=0,\\
        1, & \text{if } \alpha>0.
    \end{cases}
\end{equation}

Note that in the subsequent analysis, $s_\sharp$ plays a pivotal role in determining the qualitative behavior of solutions. The corresponding fractional Sobolev critical exponent is defined by $p_{s_\sharp}^{*}:=\frac{Np}{N-s_\sharp p}.$ Although condition \eqref{m4} does not uniquely determine $s_\sharp$, the strength of our regularity and integrability results increases with larger values of $s_\sharp$. Therefore, it is natural to choose $s_\sharp$ as large as possible while maintaining the validity of \eqref{m4}.

Throughout this work, we restrict our attention to the case $\alpha>0$ in \eqref{SP}, which corresponds to a genuinely mixed local--nonlocal superposition operator. The purely nonlocal case $\alpha=0$ requires different techniques and will be investigated in a forthcoming work.

\par In the present work, we focus on studying the regularity of weak solutions to the following problem
\begin{align}\label{MP} \tag{P}
    A_{\mu, p} u &= 0 \quad \text{in } \Omega,
\end{align}
where $p \in (1,\infty)$ and $\Omega$ is a bounded domain in $\mathbb{R}^N$.

The following examples illustrate classes of measures $\mu$ on $(0,1)$ with $\alpha>0$ that satisfy our assumptions and are therefore encompassed by the results of this paper:
\begin{enumerate}
    \item When $\mu=\delta_s$, where $\delta_s$ denotes the Dirac measure concentrated at $s\in(0,1)$, and $\alpha=1$, the operator $A_{\mu,p}$ reduces to the mixed local--nonlocal operator
$
-\Delta+(-\Delta)^s
$
in the case $p=2$, and to the mixed local--nonlocal $p$-Laplacian
$-\Delta_p+(-\Delta)_p^s$ for general $p\in(1,\infty)$.
    \item When $\mu=\epsilon\delta_s$ with $\epsilon\in(0,1)$, the operator $A_{\mu,p}$ reduces to $ -\alpha \Delta_{p} +\epsilon(-\Delta_{p})^s$.
    \item When $\mu=a_1\delta_{s_1}+a_2\delta_{s_2}$, where $a_1,a_2\in(0,+\infty)$ and $1>s_1>s_2>0$, the operator $A_{\mu,p}$ takes the form $-\alpha \Delta_p+ a_1 (-\Delta)_p^{s_1}++ a_2 (-\Delta)_p^{s_2}$.
    \item When $\mu= \sum_{k=0}^{n} a_k \delta_{s_k}$ with $n>2$, $a_k>0$, and $1>s_n>\cdots>s_1>s_0>0$, the operator $A_{\mu,p}$ is given by $-\alpha\Delta_p+\sum_{k=0}^{n} a_k (-\Delta)_p^{s_k}$.
    \item When $\mu= \sum_{k=0}^{+\infty} a_k \delta_{s_k}$ with $\sum_{k=0}^{+\infty}a_k\in(0,+\infty)$, $a_0>0$, $a_k\geq 0$, and $1>\cdots>s_2>s_1>s_0>0,$ the operator $A_{\mu,p}$ is represented by $-\alpha\Delta_p+\sum_{k=0}^{+\infty} a_k (-\Delta)_p^{s_k}$.
    \item When $\mu= \sum_{k=0}^{+\infty} a_k \delta_{s_k}$ with $\sum_{k=0}^{+\infty} a_k \in (0, +\infty)$ such that $a_0>0$ and $a_k \geq 0$ and $1 >s_0>s_1>s_2>\ldots> 0,$ then mixed local-nonlocal operators associated with a convergent series of Dirac measures is $-\alpha\Delta_p+\sum_{k=0}^{+\infty} a_k (-\Delta)_p^{s_k}$.
    
    \item When $\d\mu(s)=f(s) \d s$, where $f\not\equiv 0$ is a nonnegative measurable function, the operator $A_{\mu,p}$ becomes $-\alpha \Delta_p+\int_{(0,1)} f(s)(-\Delta)_p^s \d s$, where $\d s$ denotes the Lebesgue measure. This example incorporates both nonlinear effects and an infinite, possibly uncountable, family of fractional operators.
\end{enumerate}

Before stating our main results, we highlight the principal novelties of the present work. Recall that, in our previous paper \cite{BGKL2026}, we were unable to establish the Harnack inequality for \eqref{MP} by means of either of the superposition tails \eqref{tail1} or \eqref{tail2} associated with $A_{\mu, p}$. Moreover, we obtained the local H\"older continuity under the assumption $\bar{s}=\inf_{\supp\{\mu\}}s>0$. In this study, we obtain the local H\"older continuity without the restriction $\bar{s}>0$. To overcome this difficulty, we introduce, for the first time, the following normalized superposition tail:
\begin{equation}\label{Tailintro}
\operatorname{Tail}(u;x_0,r)
=
\left[
\int_{(0,1)}
C_{N,s,p}r^{p}
\left(
\int_{\mathbb{R}^N\setminus B_r(x_0)}
\frac{|u(y)|^{p-1}}{|y-x_0|^{N+sp}}\,dy
\right)
d\mu(s)
\right]^{\frac{1}{p-1}}.
\end{equation}
As will become evident throughout the paper, this new tail is particularly well adapted to the study of local H\"older continuity and Harnack inequalities for superposition operators. We believe that it will also prove useful in future investigations of regularity theory for mixed-order superposition operators and related nonlocal equations. A further noteworthy feature of the present work is that our results apply simultaneously to a broad class of superposition operators involving infinitely many, and possibly uncountably many, fractional phases. In particular, all of the results established here are completely new for the examples $(3)$--$(8)$ listed above. This is especially significant for examples $(6)$ and $(7)$, which were not covered by our previous work \cite{BGKL2026} and exhibit genuinely mixed local--nonlocal as well as mixed-order behavior. The absence of a single underlying differentiability scale, together with the intricate interaction among different fractional orders, makes these examples substantially more challenging than the classical local or fixed-order nonlocal settings.
Finally, one of the central contributions of this paper is the development of new techniques for handling the superposition tail \eqref{Tailintro}. The analysis of this quantity permeates the entire work and plays a decisive role in the proofs of our regularity results.

We are now ready to state the main results obtained in this paper. 
Our first result is the local H\"older continuity for weak solutions involving the ``tail" \eqref{Tailintro}.
\begin{thm}[Local H\"older continuity]\label{Holder t}
    Let $u$ be a weak solution of \eqref{MP}. Then $u$ is locally H\"older continuous in $\Omega$. Moreover, there exists $\sigma\in(0, \frac{p}{p-1})$ and positive constant $C:=C(N,p,\Sigma,\mu)$ such that
    \begin{equation*}
        \operatorname{osc}_{B_\epsilon(x_0)}u\leq C\left(\frac{\epsilon}{r} \right)^\sigma \left(  \operatorname{Tail}(u;x_0,{r})+C\left(\fint_{B_{2r}(x_0)}|u|^p \d x \right)^\frac{1}{p} \right), 
    \end{equation*}
    where $B_{2r}(x_0)\subset \Omega$ such that $r\in (0,1]$ and $\epsilon\in (0,r]$.
\end{thm}

    It is worth emphasizing that, in our previous work \cite{BGKL2026}, local H\"older continuity was established under the assumption $
\bar{s}:=\inf_{s\in\Sigma}s>0, $
with the H\"older exponent restricted to the range
$ 0<\sigma<\frac{\bar{s}p}{p-1}. $
Theorem~\ref{Holder t} significantly improves this result in two respects. First, the condition $\bar{s}>0$ is completely removed. Second, the admissible range of H\"older exponents is enlarged to
$ 0<\sigma<\frac{p}{p-1}.$
Thus, Theorem~\ref{Holder t} provides a substantially stronger regularity result in a considerably more general setting.

Our next result establishes a Harnack inequality for weak solutions of \eqref{MP}. To the best of our knowledge, this is the first Harnack inequality proved in the general framework of superposition operators and, in particular, the first result of its kind applicable to operators involving infinitely many, possibly uncountably many, fractional phases.

\begin{thm}[Harnack inequality]\label{T1.5}
Let $u$ be a weak solution of \eqref{MP} with $u\geq 0$ in $B_R(x_0)\subset \Omega$. There exists a constant $C:=C(N,p,\Sigma,\mu)>0$ such that
\begin{align*}
    \ess_{B_{\frac{r}{2}}(x_0)}u \leq C\essi_{B_r(x_0)} u + C\bigg( \frac{r}{R}\bigg)^\frac{p}{p-1}\operatorname{Tail}(u_-;x_0,R),
\end{align*}
where {$B_r(x_0) \subset B_{R}(x_0)$} with $r\in (0,1]$.
\end{thm}
\begin{cor}
   Let $u$ be a weak solution of \eqref{MP} with $u\geq 0$ in $B_R(x_0)\subset \Omega$. There exists a constant $C:=C(N,p,\Sigma,\mu)>0$ such that
\begin{align*}
   \ess_{B_{{r}}(x_0)}u \leq C\essi_{B_r(x_0)}+ C\bigg( \frac{r}{R}\bigg)^\frac{p}{p-1}\operatorname{Tail}(u_-;x_0,R),
\end{align*}
where $B_r(x_0) \subset B_{\frac{R}{2}}(x_0)$ with $r\in (0,1]$.
\end{cor}
Note that, in the case $u \geq 0$ in $\mathbb{R}^N$, the above Harnack inequality reduces to the classical Harnack inequality.

The following result concerns the weak Harnack inequality associated with \eqref{MP}, which incorporates the superposition tail \eqref{Tailintro}.

\begin{thm}[Weak Harnack inequality]\label{T1.6}
    Let $u$ be a weak supersolution of \eqref{MP} with $u\geq 0$ in $B_R(x_0) \subset \Omega$. There exists a constant $C:=C(N,p,s,\Sigma)>0$ such that
    \begin{align*}
        \bigg(\fint_{B_{\frac{r}{2}}(x_0)}u^t \d x \bigg)^\frac{1}{t} \leq C \essi_{B_r(x_0)}u +C\bigg(\frac{r}{R}\bigg)^{\frac{p}{p-1}} \operatorname{Tail}\big(u_-;x_0,R\big),
    \end{align*}
    where $r\in(0,1]$, $B_r(x_0)\subset B_{\frac{R}{2}}(x_0)$ and $0<t<\eta(p-1)$.
\end{thm}

As intermediate steps in the proof of the main results, we establish several auxiliary results, including a logarithmic estimate for weak supersolutions and a local boundedness estimate for weak subsolutions with respect to the superposition tail \eqref{Tailintro}; these are stated below. We note that analogous results were established in \cite{BGKL2026} for the superposition tail \eqref{tail2}. The present formulation extends those results to the tail \eqref{Tailintro}. Since the arguments are largely analogous, we provide only the parts of the proofs that require modification in the present setting.

\begin{lem}[Logarithmic estimate]\label{lmn3.4}
  Let $p\in (1,\infty)$ and let $u$ be a weak {supersolution} of \eqref{MP} such that $u\geq 0$ in $B_{R}:=B_R(x_0)\subset \Omega$. Then for any $B_r:=B_r(x_0)\subset B_{\frac{R}{2}}(x_0)$ and $d>0$, there exists a positive constant $C=C(N,p,\Sigma,\mu)$ such that
  \begin{align*}
     & \alpha \int_{B_r}|\nabla \log(u+d)|^p \d x+\int_{(0,1)}C_{N,p,s} \bigg(\int_{B_r}\int_{B_r}\left|\log\left(\frac{u(x)+d}{u(y)+d}\right)\right|^{p}\d \nu \bigg) \d \mu(s)\nonumber\\
        &\leq  C r^N \bigg(d^{1-p} {{R}^{-p}}[\operatorname{Tail}(u_-;x_0,R)]^{p-1}+\sup\limits_{s\in\Sigma} r^{-sp}\bigg)+C\alpha r^{N-p},
  \end{align*}
   where $\Sigma:=\operatorname{supp}\{\mu\}$.

\end{lem}
Subsequently, we establish the following local boundedness result.
\begin{thm}[Local boundedness]\label{bddness}
    Let $u$ be a weak subsolution to the problem \eqref{MP} and $p\in(1,\infty)$. There exists a constant $C:=C(N,p,\Sigma,\mu)>0$ such that
    \begin{equation}\label{LB}
        \ess_{B_{\frac{r}{2}(x_0)}} u \leq
            \delta \operatornamewithlimits{Tail}\bigg(u_+;x_0,\frac{r}{2}\bigg)+C\delta^{-\frac{(p-1)\eta}{(\eta-1)p}}\bigg( \fint_{B_r(x_0)}u_+^p \d x \bigg)^\frac{1}{p},
    \end{equation}
    where $B_r=B_r(x_0)\subset \Omega$ with $r\in (0,1]$, $\delta\in (0,1]$, $\eta$ defined in \eqref{grad sob} and $\Sigma:=\supp\{\mu\}$.
\end{thm}

The paper is organized as follows. In Section \ref{pre}, we recall several preliminary results, develop the tools needed to introduce the appropriate solution space associated with problem \eqref{MP}, and define the notion of weak solutions to \eqref{MP}. In Section \ref{sec3}, we derive a logarithmic estimate for the problem under consideration. Moreover, we prove the local boundedness of weak subsolutions to \eqref{MP}. In Section \ref{sec4}, we establish the local H\"older continuity of weak solutions to \eqref{MP}. In Section \ref{sec5}, we prove both the Harnack inequality and the weak Harnack inequality. Finally, in Section \ref{sec6}, we present some remarks and future directions towards the regularity theory for \eqref{MP}.

\section{Preliminaries} \label{pre}

This section aims to construct a functional analytic framework based on appropriate notions of fractional Sobolev spaces relavent to our problem and their properties, which are crucial for analyzing our problem. For further details on this direction, we refer to \cite{BGKL2026,DPSV1,DPSV2,DPSV} and the references therein.

\noindent We begin by introducing some notation that will be used throughout the paper.

\begin{itemize}
    \item For any $u\in \mathbb{R}$, we define $    u_+:=\max\{u,0\},
    \qquad
    u_-:=\max\{-u,0\}=-\min\{u,0\}.$ Consequently,
    $u=u_+-u_-,
    \qquad
    |u|=u_++u_-.$

    \item The symbol $\fint$ denotes the average integral over the corresponding domain.

    \item For $x_0\in\mathbb{R}^N$ and $r>0$, $B_r(x_0)$ denotes the open ball centered at $x_0$ with radius $r$. Moreover, we write $\Sigma:=\supp(\mu).$

    \item For notational convenience, we introduce the following quantities associated with the fractional part of the operator:
    \begin{align*}
        \mathcal{A}(u(x,y))
        &:=|u(x)-u(y)|^{p-2}(u(x)-u(y)),
        ~~\text{ and }~~
        \d\nu
        :=\frac{\d x\,\d y}{|x-y|^{N+sp}}.
    \end{align*}

    \item Let $\Omega\subset\mathbb{R}^N$ be a bounded domain. We denote by $W_0^{1,p}(\Omega)$ the usual Sobolev space defined by
    \begin{align*}
        W_0^{1,p}(\Omega)
        :=\left\{
        u\in W^{1,p}(\Omega):
        u=0 \text{ in } \mathbb{R}^N\setminus\Omega
        \right\}.
    \end{align*}
\end{itemize}
Recall that, for $s\in(0,1)$, the Gagliardo seminorm of a function $u$ is defined by
\begin{equation*}
    [u]_{s,p}
    :=
    \left(
    C_{N,p,s}
    \iint_{\mathbb{R}^{2N}}
    \frac{|u(x)-u(y)|^p}{|x-y|^{N+sp}}
    \,\d x\,\d y
    \right)^{\frac1p},
\end{equation*}
where
$$
C_{N,p,s}
:=
\frac{\frac{sp}{2}(1-s)2^{2s-1}}{\pi^{\frac{N-1}{2}}}
\frac{\Gamma\left(\frac{N+ps}{2}\right)}
{\Gamma\left(\frac{p+1}{2}\right)\Gamma(2-s)}
$$
is the normalization constant. Note that the choice of the constant $C_{N,p,s}$ ensures the consistency of the fractional seminorm with its limiting local and zeroth-order counterparts. More precisely,
$$
\lim_{s\searrow 0}[u]_{s,p}
=
\|u\|_{L^p(\mathbb{R}^N)}
\quad \text{and} \quad
\lim_{s\nearrow 1}[u]_{s,p}
=
\|\nabla u\|_{L^p(\mathbb{R}^N)}.
$$

Since $\mu$ is a nonnegative, nontrivial finite Borel measure on $(0,1)$, with $0<s<1$ and $1<p<\infty$, we introduce the following function space in the case $\alpha=0$:
$$
W^{s,p}_\mu(\mathbb{R}^N)
:=
\left\{
u:\mathbb{R}^N\to\mathbb{R}
\text{ measurable }:
\|u\|_\mu<+\infty
\right\},
$$
equipped with the norm
\begin{align}\label{eq2.2}
    \|u\|_\mu
    :=
    \bigg(
    \|u\|_{L^p(\mathbb{R}^N)}^p
    +
    \int_{(0,1)}
    [u]_{s,p}^p
    \,\d\mu(s)
    \bigg)^{\frac1p}.
\end{align}

We next introduce a new notion of tail, which differs from tails defined in \eqref{tail1} and \eqref{tail2}.
\begin{definition}
For a function $u\in L_{loc}^{p-1}(\mathbb{R}^N)$ and an open ball $B_r(x_0)\subset \mathbb{R}^N$, we define the nonlocal tail of $u$ with respect to the ball $B_r(x_0)$ by
\begin{equation}\label{Tail}
    \operatorname{Tail}(u;x_0,r)
    :=
    \left[
    \int_{(0,1)}
    C_{N,s,p}r^{p}
    \left(
    \int_{\mathbb{R}^N\setminus B_r(x_0)}
    \frac{|u(y)|^{p-1}}{|y-x_0|^{N+sp}}
    \,dy
    \right)
    d\mu(s)
    \right]^{\frac{1}{p-1}}.
\end{equation}
\end{definition}

We next introduce the corresponding tail space, $L^{p-1}_{sp,\mu}(\mathbb{R}^N)\supset W^{s,p}_{\mu}(\mathbb{R}^N)$, defined as
\begin{align*}
    L^{p-1}_{sp,\mu}(\mathbb{R}^N)
    :=
    \left\{
    u\in L_{loc}^{p-1}(\mathbb{R}^N):
    \operatorname{Tail}(u;x,r)<\infty,
    \ \forall\, x\in\mathbb{R}^N,\ \forall\, r\in(0,\infty)
    \right\}.
\end{align*}
We now introduce the notion of weak solutions associated with problem \eqref{MP}.
\begin{definition}\label{def2.11}
A function $u\in W_{loc}^{1,p}(\Omega)\cap L^{p-1}_{sp,\mu}(\mathbb{R}^N)$ is called a weak subsolution of \eqref{MP} if, for every $\Omega'\Subset \Omega$, the following inequality holds:
\begin{align}\label{eq2.10} 
&\int_{(0,1)}C_{N,s,p}
\left(
\iint_{\mathbb{R}^{2N}}
\frac{\mathcal{A}(u(x,y))(v(x)-v(y))}
{|x-y|^{N+sp}}
\,\d x\,\d y
\right)
\d\mu(s)
\nonumber \\ 
&\quad+\alpha\int_{\Omega}
|\nabla u(x)|^{p-2}\nabla u(x)\cdot\nabla v(x)
\,\d x
\leq 0, ~\text{for all nonnegative }v\in W_0^{1,p}(\Omega').
\end{align}

Similarly, a function $u\in W_{loc}^{1,p}(\Omega)\cap L^{p-1}_{sp,\mu}(\mathbb{R}^N)$ is called a weak supersolution of \eqref{MP} if the left-hand side of \eqref{eq2.10} is nonnegative for every $\Omega'\Subset\Omega$ and every nonnegative function $v\in W_0^{1,p}(\Omega')$. Finally, a function $u\in W_{loc}^{1,p}(\Omega)\cap L^{p-1}_{sp,\mu}(\mathbb{R}^N)$ is called a weak solution of \eqref{MP} if the left-hand side of \eqref{eq2.10} vanishes for every $v\in W_0^{1,p}(\Omega')$.
\end{definition}
Observe that Definition \ref{def2.11} immediately yields the following elementary properties:
\begin{itemize}
    \item[(i)] $u$ is a weak subsolution of \eqref{MP} if and only if $-u$ is a weak supersolution of \eqref{MP}$.$

    \item[(ii)] For any $c\in\mathbb{R}$, the function $u+c$ is a weak solution of \eqref{MP} if and only if $u$ is a weak solution of \eqref{MP}$.$

    \item[(iii)] $u$ is a weak solution of \eqref{MP} if and only if $-u$ is a weak solution of \eqref{MP}$.$
\end{itemize}
In the following lemma, we collect several results from \cite{BGKL2026} that will be used repeatedly throughout the paper.

\begin{lem}\label{nw}
The following statements hold:
\begin{enumerate}
    \item[(a)] The space $W^{s,p}_\mu(\mathbb{R}^N)$ is a Banach space when endowed with the norm \eqref{eq2.2}.

    \item[(b)] Let $\Omega$ be a bounded domain in $\mathbb{R}^N$, $s\in(0,1)$, and $p\in(1,\infty)$. Then there exists a constant $C=C(N,\Omega,p)>0$ such that
    \begin{align*}
        \int_{(0,1)}[u]_{s,p}^p\,\d\mu(s)
        \leq
        C\int_\Omega |\nabla u(x)|^p\,\d x,
        \qquad \forall\,u\in W_0^{1,p}(\Omega).
    \end{align*}

    \item[(c)] A function $u$ is a weak solution of \eqref{MP} if and only if it is both a weak subsolution and a weak supersolution of \eqref{MP}.

    \item[(d)] If $u$ is a weak subsolution of \eqref{MP}$,$ then $u_+$ is a weak subsolution of \eqref{MP}$.$

    \item[(e)] If $u$ is a weak supersolution of \eqref{MP}$,$ then $u_-$ is a weak subsolution of \eqref{MP}$.$

    \item[(f)] Let $u$ be a weak subsolution of \eqref{MP} and let $\phi=(u-k)_+$ for some $k\in\mathbb{R}$. Then there exists a positive constant $C:=C(p)$ such that
    \begin{align}\label{eq3.1}
     &\alpha \int_{B_r(x_0)} w^p|\nabla \phi|^p \d x+ \int_{(0,1)} C_{N,s,p} \left( \iint_{B_r(x_0)\times B_r(x_0)} \frac{|\phi(x) w(x)-\phi(y)w(y)|^p}{|x-y|^{N+sp}}\d x \d y\right) \d \mu \nonumber\\
    \leq &  C \Bigg[ \int_{(0,1)} C_{N,s,p} \left( \iint_{B_r(x_0)\times B_r(x_0)} \frac{ \max\{\phi(x), \phi(y)\}^p|w(x)-w(y)|^p}{|x-y|^{N+sp}}\d x\d y\right)\d \mu \nonumber \\
  & \quad\quad+ \int_{(0,1)} C_{N,s,p} \left(\underset{x\in \supp \{w\}}{\ess} \int_{\mathbb{R}^N\setminus B_r(x_0)} \frac{\phi^{p-1}(y)}{|x-y|^{N+sp}}\d y \cdot \int_{B_r(x_0)}\phi(x) w^p(x) \d x \right)\d \mu \nonumber\\
  &\quad\quad\quad\quad\quad+\alpha \int_{B_r(x_0)} \phi^p|\nabla w|^p \d x \Bigg],
    \end{align}
    where $w$ is a nonnegative function satisfying $w\in C_c^\infty(B_r(x_0))$ and $B_r(x_0)\subset\Omega$.

    \item[(g)] Let $u$ be a weak supersolution of \eqref{MP} and let $\phi=(u-k)_-$ for some $k\in\mathbb{R}$. Then estimate \eqref{eq3.1} remains valid.

    \item[(h)] Let $u$ be a weak solution of \eqref{MP} and let $\phi=(u-k)_\pm$ for some $k\in\mathbb{R}$. Then estimate \eqref{eq3.1} holds.
\end{enumerate}
\end{lem}

We conclude this section by recalling several auxiliary results that will play a crucial role in the proofs of the H\"older continuity and Harnack inequality.

\begin{lem}[Corollary $1.57$, \cite{MZ97}]
Let $\Omega$ be an open bounded subset of $\mathbb{R}^N$, $p\in(1,\infty)$, and define
\begin{align}\label{grad sob}
    \eta=
    \begin{cases}
        \dfrac{N}{N-p}, & \text{when } p\in(1,N),\\
        2, & \text{when } p\in[N,\infty).
    \end{cases}
\end{align}
Then there exists a constant $C=C(p,N)>0$ such that
\begin{align}\label{sob}
    \left(
    \int_\Omega |u(x)|^{\eta p}\,\d x
    \right)^{\frac{1}{\eta p}}
    \leq
    C|\Omega|^{\frac{1}{N}-\frac{1}{p}+\frac{1}{\eta p}}
    \left(
    \int_\Omega |\nabla u(x)|^p\,\d x
    \right)^{\frac{1}{p}},~~\text{for all }u\in W_0^{1,p}(\Omega).
\end{align}
\end{lem}

\begin{lem}[Lemma $4.1$, \cite{D2012}]\label{MI Lemma}
Let $\{P_j\}_{j=0}^{\infty}$ be a sequence of positive real numbers satisfying $P_{j+1}\leq C_1C_2^jP_j^{1+\beta},$
where $C_1,C_2>1$ and $\beta>0$ are constants. If $
P_0\leq C_1^{-\frac{1}{\beta}}C_2^{-\frac{1}{\beta^2}},$
then we have $\lim_{j\to\infty}P_j=0.$
\end{lem}

\begin{lem}[Lemma $2.5$, \cite{DKP2014}]\label{lmn6.30}
Let $A\subset B_r(x_0)$ be a measurable set with $r\in(0,1]$. For $\delta'\in(0,1)$, define
\begin{align*}
    A_{\delta'}
    :=
    \bigcup_{\gamma>0}
    \bigg\{
    B_{3\gamma}(x)\cap B_r(x_0):
    x\in B_r(x_0)
    \, \text{and}\,\,
    |A\cap B_{3\gamma}(x)|
    >
    \delta'|B_{\gamma}(x)|
    \bigg\}.
\end{align*}
Then either $|A_{\delta'}|
\geq
\frac{C}{\delta'}|A|,$ or $A_{\delta'}=B_r(x_0),$ where $C=C(N)\in(0,1]$.
\end{lem}

\section{Logarithmic estimate and Local  Boundedness} \label{sec3}
In this section, we establish a significant logarithmic estimate and the local boundedness involving a new tail \eqref{Tail}, which will be utilized in the following section.

\subsection{Proof of Lemma \ref{lmn3.4}}
\begin{proof} 
The proof follows the same lines as that of \cite[Lemma $5.1$]{AGKR2025i}, except for a crucial estimate involving the tail term. Therefore, we only discuss this part here and refer the reader to \cite[Lemma $5.1$]{AGKR2025i} for the remaining details of the proof.

Let $d>0$ and $\eta \in C_c^{\infty}({\mathbb{R}^N})$ be such that \begin{equation*}
0\leq\eta\leq1,~~~\eta\equiv1~\text{in}~B_r:=B_r(x_0),~~~\eta\equiv0~\text{in}~{\mathbb{R}^N}\setminus B_{\frac{3r}{2}}(x_0)~~\text{and}~|\nabla \eta|<Cr^{-1}.
\end{equation*}
Let us choose  $v=(u+d)^{1-p}\eta^p$ is a well-defined test function of \eqref{MP}.

The proof then proceeds {\it verbatim} as in \cite[Lemma $5.1$]{AGKR2025i}, except for the final part of the estimate of $\mathbf{I_4}$. Hence, we only provide this modified estimate below. Incorporating this estimate into the argument of \cite[Lemma $5.1$]{AGKR2025i}, we obtain the desired logarithmic estimate.
\noindent For ${x\in B_{\frac{3r}{2}}:=B_{\frac{3}{2}r}}(x_0)$ and $y\in\mathbb{R}^N\setminus B_{2r}(x_0) $, we get
\begin{align}\label{eq3.3}
   \frac{1}{ |y-x|}\leq \frac{2}{r}, \quad \text{ and } \quad  \frac{|y-x_0|}{|y-x|} \leq 1 +\frac{|x-x_0|}{|y-x|} \leq 1+\frac{\frac{3}{2}r}{\frac{1}{2}r}=4.
\end{align}

Moreover, observe that $u(y)\geq 0$ for all  $y \in B_{R}$. Thus, using $(u(x)-u(y))_{+}\leq u(x)$, we get
\begin{equation}\label{4.40}
	\frac{(u(x)-u(y))_{+}^{p-1}}{(d+u(x))^{p-1}} \leq 1,~\forall\, x \in B_{2 r}, \,y \in B_{R}.
\end{equation}
On the other hand, for all $x\in B_{2r}$ and $y \in \mathbb{R}^N\setminus B_{R}$, we have
\begin{equation}\label{4.41}
	(u(x)-u(y))_{+}^{p-1} \leq 2^{p-1}\left[u^{p-1}(x)+(u(y))_{-}^{p-1}\right].
\end{equation}
\textbf{New estimate for a part of estimate of $I_4$ in  \cite[Lemma $5.1$]{AGKR2025i}:} Using \eqref{eq3.3}, \eqref{4.40} and \eqref{4.41}, we derive
   \begin{align*}
		I_{4}  \leq & 2 \int_{(0,1)} C_{N, p, s}\bigg( \int_{B_{R} \setminus B_{2 r}} \int_{B_{2 r}} (u(x)-u(y))_{+}^{p-1}(d+u(x))^{1-p} \frac{ \eta^{p}(x)}{|x-y|^{N+ps}}\d x \d y \bigg) \d \mu(s)\nonumber \\
		 &+2\int_{(0,1)} C_{N, p, s} \bigg(\int_{{\mathbb{R}^N} \setminus B_{R}} \int_{B_{2 r}} (u(x)-u(y))_{+}^{p-1}(d+u(x))^{1-p} \frac{ \eta^{p}(x)}{|x-y|^{N+ps}}\d x\d y\bigg) \d \mu(s)\nonumber \\
		\leq & C\int_{(0,1)} C_{N, p, s} \bigg(\int_{{\mathbb{R}^N}\setminus B_{2 r}} \int_{B_{2 r}} \frac{\eta^{p}(x)}{|x-y|^{N+ps}}\d x \d y \bigg) \d \mu(s)\nonumber\\
        &+Cd^{1-p} \int_{(0,1)} C_{N, p, s} \bigg(\int_{{\mathbb{R}^N} \setminus B_{R}} \int_{B_{2 r}} \frac{(u(y))_{-}^{p-1}\eta^p(x)}{|x-y|^{N+ps}} \d x \d y \bigg) \d \mu(s)\nonumber\\
		\leq& {C}  r^{N} \int_{(0,1)} C_{N, p, s}\bigg( \sup _{x \in B_{\frac{3r}{2}}} \int_{{\mathbb{R}^N}\setminus B_{2 r}} \frac{1}{|x-y|^{N+ps}}\d y \bigg) \d\mu(s)\nonumber\\
        &+Cd^{1-p} \int_{(0,1)} C_{N, p, s} \bigg( \int_{{\mathbb{R}^N} \setminus B_{2r}}\int_{B_{\frac{3r}{2}}} {\frac{(u(y))_{-}^{p-1}}{\left|x-y\right|^{N+ps}}\d x \d y} \bigg) \d \mu(s)\nonumber\\
        \leq& C \int_{(0,1)} C_{N, p, s} \frac{2^{N+sp}}{s}{r^{N-ps}}\d \mu(s)\nonumber\\
        &+Cd^{1-p}\left|B_{\frac{3r}{2}}\right| \int_{(0,1)} C_{N, p, s} \bigg( \int_{{\mathbb{R}^N} \setminus B_{R}} {\frac{4^{N+sp}(u(y))_{-}^{p-1}}{\left|x_{0}-y\right|^{N+ps}}\d y} \bigg) \d \mu(s)\nonumber\\
		\leq& C  \sup_{s\in \Sigma} 2^{N+sp}\sup_{s \in \Sigma} r^{N-sp} + Cd^{1-p} r^{N} \sup_{s\in \Sigma} 4^{N+sp} R^{-p}\left[\operatorname{Tail}\left(u_{-} ; x_{0}, R\right)\right]^{p-1}\nonumber\\
		\leq & C  \sup_{s \in \Sigma} r^{N-sp}  + Cd^{1-p} r^{N}  R^{-p}  \left[\operatorname{Tail}\left(u_{-} ; x_{0}, R\right)\right]^{p-1}.
	\end{align*}
    This completes the proof.
\end{proof}

\begin{cor}\label{coro}
    Suppose that $u$ is a weak solution of \eqref{MP} such that $u\geq 0$ in $B_R:=B_R(x_0) \subset \Omega$ and $b,d>0,~ a>1$. Let us define
    \begin{align}\nonumber
        \phi= \min\left\{\left(\log\left(\frac{b+d}{u+d}\right)\right)_+, \ \log a\right\}.
    \end{align}
    Then there exists a constant $C=C(N,p, \Sigma, \mu)>0$ such that 
    \begin{align*}
        \fint_{B_r(x_0)}\left|\phi-(\phi)_{B_r(x_0)}\right|^p \d x \leq C  \bigg(d^{1-p} \left(\frac{r}{R}\right)^{p}[\operatorname{Tail}(u_-;x_0,R)]^{p-1}+1\bigg),
    \end{align*}
    where $B_r:=B_r(x_0)\subset B_{\frac{R}{2}}(x_0)$ with $r\in(0,1]$, $(\phi)_{B_r(x_0)}=\frac{1}{|B_r(x_0)|}\int_{B_r(x_0)}\phi(x)\d x$ and $\Sigma:=\operatorname{supp}\{\mu\}$.
\end{cor}
\begin{proof}
    Using the classical Poinca\'re inequality \cite[Theorem 2]{EVANS2022}, we get
    \begin{equation}\label{eq3.9}
        \fint_{B_r(x_0)}\left|\phi-(\phi)_{B_r(x_0)}\right|^p \d x \leq Cr^{p-N}\int_{B_r(x_0)} |\nabla \phi|^p \d x.
    \end{equation}
    Since $\phi$ is a truncation of the sum of $\log(u+d)$ and constant, then we obtain
    \begin{align}\label{eq3.10}
        \int_{B_r(x_0)} |\nabla \phi|^p \d x \leq  \int_{B_r(x_0)} |\nabla \log(u+d)|^p \d x.
    \end{align}
    Using \eqref{eq3.9}, \eqref{eq3.10} and the Logarithmic estimate from Lemma \ref{lmn3.4}, we derive
    \begin{align*}\nonumber
        \fint_{B_r(x_0)}\left|\phi-(\phi)_{B_r(x_0)}\right|^p \d x \leq & Cr^{p-N} \left(\alpha \int_{B_r(x_0)}|\nabla \log(u+d)|^p \d x \right) \\ \nonumber
        \leq &  C r^p\bigg(d^{1-p} {{R}^{-p}}[\operatorname{Tail}(u_-;x_0,R)]^{p-1}+\sup\limits_{s\in\Sigma} r^{-sp}\bigg)+C\alpha \nonumber \\
        \leq & C \bigg(d^{1-p} \left(\frac{r}{R}\right)^{p}[\operatorname{Tail}(u_-;x_0,R)]^{p-1}+1\bigg)+C\alpha \\ \nonumber
         \leq & C \bigg(d^{1-p} \left(\frac{r}{R}\right)^{p}[\operatorname{Tail}(u_-;x_0,R)]^{p-1}+1\bigg). 
    \end{align*}
    This completes the proof.
\end{proof}

Next, we prove local boundedness for the weak subsolution of \eqref{MP} using new tail \eqref{Tail}

\subsection{Proof of Theorem \ref{bddness}}

\begin{proof}
The proof follows the same approach as that of \cite[Lemma 4.1]{BGKL2026}, except for the estimate of \textbf{$I_2$} involving tail \eqref{tail2}. Therefore, we only present the modified estimate for $I_2$ here. Replacing this estimate into the argument of \cite[Lemma 4.1]{BGKL2026}, we obtain the desired conclusion.

    For $j\in \mathbb{N}\cup\{0\} $ and $r\in (0,1]$ with $B_r=B_r(x_0)\subset \Omega$, we define
    $$r_j=\frac{r}{2}\left(1+\frac{1}{2^{j}}\right),~\bar{r_j}=\frac{r_j+r_{j+1}}{2},~ B_j=B_{r_j}(x_0) \text{ and } \bar{B_j}=B_{\bar{r_j}}(x_0).$$
    Observer that $r_{j+1}<\bar{r_j}<r_j$ and $B_{j+1}\subset\bar{B_j}\subset B_j$. Again, for $\bar{k}>0$ and $k\in \mathbb{R}$, we define
    $$k_j=k+\left( 1-\frac{1}{2^j}\right)\bar{k},~ \bar{k_j}=\frac{k_j+k_{j+1}}{2},~ \bar{\phi_j}=(u-\bar{k_j})_+ \text{ and } {\phi_j}=(u-{k_j})_+.$$
     Note that $k_{j}<\bar{k_j}<k_{j+1}$ and $\phi_{j+1}\leq \bar{\phi_j}\leq \phi_j.$ Let $w_j\in C_c^\infty(\bar{B_j})$ such that $0\leq w_j\leq 1$ in $\bar{B_j}$,
     $w_j=1$ in $B_{j+1}$ and $|\nabla w_j|\leq \frac{2^{j+3}}{r}$. 
     Moreover, $k_{j+1}-\bar{k_j}=\bar{k_j}-k_j=\frac{\bar{k}}{2^{j+1}}$. 
    
 \noindent For $x\in \supp \{w_j\}=\bar{B_j}$ and $y\in \mathbb{R}^N\setminus B_j$, we have
\begin{align}\label{3.12}
    \frac{|y-x_0|}{|y-x|}\leq \frac{|y-x|+|x-x_0|}{|y-x|}\leq 1+\frac{\bar{r_j}}{r_j-\bar{r_j}} \leq 2^{j+4}.
\end{align}
furthermore, \begin{equation} \label{3.13}
\phi_j^p\geq(\bar{k_j}-k_j)^{p-1}\bar{\phi_j}.
\end{equation}
\textbf{New estimate of $I_2$:}  Using \eqref{3.12} and \eqref{3.13}, we obtain
\begin{align*}
     I_2=&C r^{p-N} \int_{(0,1)} C_{N,s,p} \left( \ess_{x\in \supp \{w_j\}} \int_{\mathbb{R}^N\setminus B_j} \frac{\bar{\phi_j}^{p-1}(y)}{|y-x|^{N+sp}}\d y  \int_{B_j}\bar{\phi_j} w^p \d x \right)\d \mu \nonumber\\
     \leq & C \int_{(0,1)}  C_{N,s,p}  \bigg(r^p 2^{(j+4)(N+sp)} \int_{\mathbb{R}^N\setminus B_j} \frac{{\phi_j}^{p-1}(y)}{|y-x_0|^{N+sp}}\d y \fint_{B_j} \frac{\phi^p_j(x)}{(\bar{k_j}-k_j)^{p-1}}\d x \bigg)\d \mu \nonumber \\
     \leq & C \int_{(0,1)}C_{N,s,p} \bigg\{ r^p \frac{2^{j(N+sp)+4sp+j(p-1)}}{\bar{k}^{p-1}} \int_{\mathbb{R}^N\setminus B_{\frac{r}{2}}} \frac{{\phi_j}^{p-1}(y)}{|y-x_0|^{N+sp}}\d y \left(\fint_{B_j} \phi_j^p(x) \d x\right)\bigg\}\d \mu \nonumber\\
     = &C_1 {\int_{(0,1)}  C_{N,s,p} \bigg\{ \frac{2^{j(N+sp+p-1)+4sp}}{\bar{k}^{p-1}} \left(\frac{r}{2}\right)^{p} \int_{\mathbb{R}^N\setminus B_{\frac{r}{2}}} \frac{{\phi_0}^{p-1}(y)}{|y-x_0|^{N+sp}}\d y \left(\fint_{B_j} \phi_j^p(x) \d x\right) \bigg\} \d \mu }\nonumber\\
    \leq & C {\left(\fint_{B_j} \phi_j^p(x) \d x\right)  \sup_{s\in \Sigma}\{2^{4sp} \} \cfrac{\sup_{s\in \Sigma}\{2^{j(N+sp+p-1)}\}}{\bar{k}^{p-1}} \bigg[\operatorname{Tail}\bigg(\phi_0;x_0,\frac{r}{2}\bigg)\bigg]^{p-1} }   \nonumber \\
    \leq & C 2^{j(N+p-1)+\sup_{s\in \Sigma}\{spj\}} \delta^{1-p}\left(\fint_{B_j} \phi_j^p(x) \d x\right), 
\end{align*}
where $C=C(N,p,\Sigma,\mu)>0$ and $\delta\operatornamewithlimits{Tail}(\phi_0;x_0,\frac{r}{2})\leq \bar{k}$ with  $\delta\in(0,1]$.
This completes the proof. \end{proof}

\section{Local H\"older continuity} \label{sec4}
In this section, we establish the local H\"older continuity by means of the new superposition tail \eqref{Tail}, without requiring the condition
$\bar{s}=\inf_{s\in\Sigma}s>0$ that was imposed in \cite{BGKL2026} when working with the tail quantity \eqref{tail2}. Our approach follows the method developed by Di Castro \textit{et al.} \cite[Lemma $5.1$]{DKP2016}, combined with the iteration Lemma \ref{MI Lemma}.

\begin{lem}\label{lmn5.1}
    Let $u$ be a weak solution of \eqref{MP} and $B_R(x_0)\subset \Omega$ for some $R$ with $0<r<\frac{R}{2}$, $r\in(0,1].$ Let $\kappa\in(0,\frac{1}{4}]$, we set $r_j=\kappa^j\frac{r}{2}$ and $B_j:=B_{r_j}(x_0)$ for $j\in \mathbb{N} \cup \{0\}$. We denote
    \begin{equation}\label{eq5.1}
        \frac{1}{2}\phi(r_0)=\operatorname{Tail}\bigg(u;x_0,\frac{r}{2}\bigg)+C\left(\fint_{B_r(x_0)}|u|^p \d x \right)^\frac{1}{p}, 
    \end{equation}
     and 
     \begin{equation*}
         \phi(r_j)=\left(\frac{r_j}{r_0}\right)^\sigma \phi(r_0), ~j\in \mathbb{N}, 
     \end{equation*}
     where {$\sigma \in(0,\frac{p}{p-1})$} and $C:=C(N,p,\Sigma,\mu)$ is the same constant as in \eqref{LB}. Then 
     \begin{equation}\label{eq5.3}
         \operatorname{osc}_{B_j} u=: \ess_{B_j} u -\essi_{B_j} u \leq \phi(r_j),~ j\in \mathbb{N} \cup \{0\}.
     \end{equation}
\end{lem}
\begin{proof} 
By applying $(d)$ and $(e)$ of Lemma \ref{nw}, we observe that both $u_+$ and $(-u)_+=u_-$ are weak subsolutions. From Theorem \ref{bddness}, the estimate \eqref{eq5.3} follows for $j=0$. We adopt induction principle to prove \eqref{eq5.3}. For this, let us assume that \eqref{eq5.3} holds for $i\in 0,1,2,....,j$ for some $j\geq 0$. It remains to obtain the estimate \eqref{eq5.3} for $i=j+1$. 
Note that either one of the following two estimates is true:
\begin{align}\label{eq5.4}
    \frac{|\{u\geq \essi_{B_j}u+\frac{\phi(r_j)}{2}\} \cap 2B_{j+1}|}{| 2B_{j+1}|}\geq \frac{1}{2},
\end{align}
or
\begin{align}\label{eq5.5}
    \frac{|\{u\leq \essi_{B_j}u+\frac{\phi(r_j)}{2}\} \cap 2B_{j+1}|}{| 2B_{j+1}|}\geq \frac{1}{2}.
\end{align}
 We set
\begin{align*}
    u_j:=\begin{cases}
        & u- \essi_{B_j} u,  ~~~~~~~~~~~~~\text{ when } \eqref{eq5.4} \text{ holds},\\
        &\phi(r_j)-(u- \essi_{B_j} u), \text{ when } \eqref{eq5.5} \text{ holds}.
    \end{cases}
\end{align*}
Then,  we deduce that $u_j$ is a weak solution satisfying
\begin{enumerate}
    \item[(i)] $u_j\geq 0$ in $B_j$,
    \item[(ii)] $ \frac{|\{u_j\geq \frac{\phi(r_j)}{2}\} \cap 2B_{j+1}|}{| 2B_{j+1}|}\geq \frac{1}{2},$
    \item[(iii)] {$\ess_{B_i}|u_j|\leq 2 \phi(r_i)$  $\forall ~i\in  \{0,1,2,...,j\}.$}
\end{enumerate}
Next, we establish the following estimate: 
\begin{align}\label{eq5.7}
    [\operatorname{Tail}(u_j;x_0,r_j)]^{p-1} \leq C \kappa^{-\sigma(p-1)}[\phi(r_j)]^{p-1},
\end{align}
 where $C=C(N,p,\mu,\Sigma,|\frac{p}{p-1}-\sigma|)$ a positive constant. Indeed, we obtain
\begin{align}\label{eq5.8}
   [\operatorname{Tail}(u_j;x_0,r_j)]^{p-1}=&\int_{(0,1)}{C_{N,p,s}}r_j^{p}\left(\int_{\mathbb{R}^N\setminus B_{r_j}(x_0)}\frac{|u_j(x)|^{p-1}}{|x-x_0|^{N+sp}}\d x\right)\d \mu(s) \nonumber\\
   =& \underbrace{\int_{(0,1)}C_{N,p,s} r_j^{p} \bigg( \sum_{i=1}^j \int_{B_{i-1}\setminus B_i} \frac{|u_j(x)|^{p-1}}{|x-x_0|^{N+sp}}\d x \bigg)\d \mu(s)}_{J_1} \nonumber \\
   &+ \underbrace{\int_{(0,1)}C_{N,p,s} r_j^{p} \bigg( \int_{\mathbb{R}^N\setminus B_0} \frac{|u_j(x)|^{p-1}}{|x-x_0|^{N+sp}}\d x \bigg)\d \mu(s)}_{J_2}.  
   \end{align}
   By using the condition (iii) and  \eqref{eq5.1}, we estimate
\begin{align}\label{eq5.9}
   J_1=& \int_{(0,1)}C_{N,p,s} r_j^{p} \left( \sum_{i=1}^j \int_{B_{i-1}\setminus B_i} \frac{|u_j(x)|^{p-1}}{|x-x_0|^{N+sp}}\d x \right)\d \mu(s) \nonumber\\
   \leq & \int_{(0,1)}C_{N,p,s} r_j^{p} \left( \sum_{i=1}^j\ess_{B_{i-1}}|u_j|^{p-1} {{\int_{\mathbb{R}^N\setminus B_i}} \frac{1}{|x-x_0|^{N+sp}}\d x} \right) \d \mu(s)\nonumber \\
   \leq & C { \int_{(0,1)} {C_{N,p,s}\frac{1}{s}} r_j^p\sum_{i=1}^j\left(\frac{1}{r_i}\right)^{sp}\phi(r_{i-1})^{p-1}d \mu(s)}\nonumber \\
    \leq & C { \int_{(0,1)} {C_{N,p,s}\frac{1}{s}} \sum_{i=1}^j\left(\frac{r_j}{r_i}\right)^{p}\phi(r_{i-1})^{p-1}d \mu(s)}\nonumber \\
    = & C  \int_{(0,1)} C_{N,p,s}\frac{1}{s}\phi(r_{0})^{p-1}\left( \frac{r_j}{r_0}\right)^{\sigma(p-1)}  \sum_{i=1}^j\left(\frac{r_{i-1}}{r_i}\right)^{\sigma(p-1)}\left( \frac{r_j}{r_i}\right)^{p-\sigma(p-1)}\d \mu(s)\nonumber \\
    = & C  \int_{(0,1)} C_{N,p,s}\frac{1}{s} \phi(r_{j})^{p-1}  \sum_{i=1}^j \kappa^{-\sigma(p-1)} \kappa^{(j-i)\{p-\sigma(p-1)\}}  \d  \mu(s)  \nonumber \\ 
     \leq & C \phi(r_{j})^{p-1} \kappa^{-\sigma(p-1)}  \sum_{i=1}^j  \kappa^{(j-i)\{p-\sigma(p-1)\}} \mu\{(0,1)\}, \nonumber \\
     \leq & C \phi(r_{j})^{p-1} \kappa^{-\sigma(p-1)}  \sum_{i=0}^{j-1}  \kappa^{i\{p-\sigma(p-1)\}}. 
\end{align}
Similarly, we have the following estimate for $J_2:$
\begin{align}\label{eq5.100}
    J_2=& \int_{(0,1)} C_{N,p,s}r_j^{p} \left( \int_{\mathbb{R}^N\setminus B_0} \frac{|u_j(x)|^{p-1}}{|x-x_0|^{N+sp}}\d x \right)\d \mu(s)\nonumber\\
    \leq & C\int_{(0,1)}C_{N,p,s} r_j^{p} \left( \int_{\mathbb{R}^N\setminus B_0} \frac{|u(x)|^{p-1}+\phi^{p-1}(r_0)+\sup_{B_0}|u|^{p-1}}{|x-x_0|^{N+sp}}\d x \right)\d \mu(s)\nonumber\\
    \leq & C\int_{(0,1)}C_{N,p,s} \left(\frac{r_j}{r_0}\right)^{p}\left[ {\frac{1}{s}}\{\phi^{p-1}(r_0)+\sup_{B_0}|u|^{p-1}\} + r_0^{p} \int_{\mathbb{R}^N\setminus B_0} \frac{|u(x)|^{p-1}}{|x-x_0|^{N+sp}}\d x\right] \d \mu(s)\nonumber\\
     \leq & C\int_{(0,1)}C_{N,p,s} \left(\frac{r_j}{r_1}\right)^{p}\left[{\frac{1}{s}}\phi^{p-1}(r_0) + r_0^{p} \int_{\mathbb{R}^N\setminus B_0} \frac{|u(x)|^{p-1}}{|x-x_0|^{N+sp}}\d x\right] \d \mu(s)\nonumber\\
 = & C\int_{(0,1)} C_{N,p,s}\left(\frac{r_j}{r_0}\right)^{\sigma(p-1)} \kappa^{(j-1)\{p-\sigma(p-1)\}} \kappa^{-\sigma(p-1)} \bigg[{\frac{1}{s}}\phi^{p-1}(r_0) \nonumber \\
& \quad\quad\quad\quad+ r_0^{p} \int_{\mathbb{R}^N\setminus B_0} \frac{|u(x)|^{p-1}}{|x-x_0|^{N+sp}}\d x\bigg] \d \mu(s)\nonumber\\
     \leq & C \left(\frac{r_j}{r_0}\right)^{\sigma(p-1)}  \kappa^{(j-1)\{p-\sigma(p-1)\}} \kappa^{-\sigma(p-1)} \left[\phi^{p-1}(r_0) + [\operatorname{Tail}(u;x_0,r_0)]^{p-1} \right] \mu\{(0,1)\} \nonumber \\
     \leq & C \left(\frac{r_j}{r_0}\right)^{\sigma(p-1)} \phi^{p-1}(r_0) \kappa^{(j-1)\{p-\sigma(p-1)\}} \kappa^{-\sigma(p-1)}\nonumber \\
     = & C \phi(r_j)^{p-1} \kappa^{-\sigma(p-1)} \kappa^{(j-1)\{p-\sigma(p-1)\}}. 
\end{align}

Now, utilizing estimates \eqref{eq5.9} and \eqref{eq5.100}, in \eqref{eq5.8}, we have
   \begin{align}
  [\operatorname{Tail}(u_j;x_0,r_j)]^{p-1}  \leq & C {\phi(r_{j})^{p-1} \kappa^{-\sigma(p-1)}  \sum_{i=0}^{j-1} \kappa^{i\{p-\sigma(p-1)\}} }  \nonumber \\
    \leq  & C  \phi(r_{j})^{p-1} \frac{\kappa^{-\sigma(p-1)}}{1-\kappa^{p-\sigma(p-1)}} \nonumber \\ 
     \leq  & C \frac{4^{p-\sigma(p-1)}}{\log4 \{p-\sigma(p-1)\}} \kappa^{-\sigma(p-1)} \phi(r_{j})^{p-1}= C \kappa^{-\sigma(p-1)} \phi(r_{j})^{p-1},
\end{align}
where we have used $0<\kappa\leq \frac{1}{4}$, $0<\sigma<\frac{p}{p-1}.$ 

Now, the proof of \eqref{eq5.3} is divided into two part.

\textbf{Step 1:-} In the first step, we show that 
\begin{align*}
    \frac{|\{u_j \leq 2 \tau\phi(r_j)\}\cap 2B_{r_{j+1}}(x_0)|}{|2B_{j+1}(x_0)|} \leq \frac{\bar{C}}{\log\frac{1}{\kappa}},
\end{align*}
where $\tau=\kappa^{\frac{p}{p-1}-\sigma}$ for some positive constant $\bar{C}:=\bar{C}(N,p,\Sigma,\mu,|\frac{p}{p-1}-\sigma|)$. We denote
\begin{align*}
    \epsilon:=\log\left( \frac{\frac{\phi(r_j)}{2}+\tau\phi(r_j)}{3\tau \phi(r_j)}\right)=\log\left(\frac{\frac{1}{2}+\tau}{3\tau}\right) \sim \log \left( \frac{1}{\tau} \right), \text{ as } \tau \rightarrow 0,
\end{align*}
and 
\begin{align*}
    q:=\min\left\{ \left( \log\left( \frac{\frac{\phi(r_j)}{2}+\tau\phi(r_j)}{u_j+\tau \phi(r_j)}\right) \right)_+, \epsilon  \right\}.
\end{align*}
Using the condition $(ii)$, we get
\begin{align}\label{eq5.13}
    \epsilon= &\frac{1}{|\{u_j\geq \frac{\phi(r_j)}{2}\} \cap 2B_{j+1}|}\int_{\{u_j\geq \frac{\phi(r_j)}{2}\} \cap 2B_{j+1}} \epsilon \d x \nonumber \\
     = &\frac{1}{|\{u_j\geq \frac{\phi(r_j)}{2}\} \cap 2B_{j+1}|}\int_{\{q=0\} \cap 2B_{j+1}} \epsilon \d x \nonumber \\
     \leq &\frac{2}{| 2B_{j+1}|}\int_{ 2B_{j+1}}\left(\epsilon-q\right) \d x=2\left(\epsilon-(q)_{ 2B_{j+1}}\right),
\end{align}
where $(q)_{ 2B_{j+1}}=\fint_{{ 2B_{j+1}}}q \d x$. Integrating both side of \eqref{eq5.13} over $[{\{q=\epsilon\} \cap 2B_{j+1}}]$, we obtain
\begin{align}\label{eq5.14}
    \frac{|{\{q=\epsilon\} \cap 2B_{j+1}}|\epsilon}{|2B_{j+1}|} \leq & \frac{2}{|2B_{j+1}|} \int_{[{\{q=\epsilon\} \cap 2B_{j+1}}]} \left(\epsilon-(q)_{ 2B_{j+1}}\right) \d x \nonumber \\
    \leq & \frac{2}{|2B_{j+1}|} \int_{\{ 2B_{j+1}\}} |q-(q)_{ 2B_{j+1}}| \d x. 
\end{align}
By using Corollary \ref{coro} along with $b=\frac{\phi(r_j)}{2}$, $d=\tau \phi(r_j)$, $a=e^\epsilon$, and {$B_{j+1}\subset  B_{\frac{j}{2}}$} then there exists a constant $C>0$, such that
\begin{align}\label{eq5.16}
    \fint_{\{ 2B_{j+1}\}} |q-(q)_{ 2B_{j+1}}|^p \d x \leq & C  \bigg((\tau \phi(r_j))^{1-p} \left(\frac{r_{j+1}}{r_j}\right)^{p}[\operatorname{Tail}((u_j)_-;x_0,2r_j)]^{p-1}+1\bigg), \nonumber \\
    \leq & C  \bigg((\tau \phi(r_j))^{1-p} \left(\frac{r_{j+1}}{r_j}\right)^{p}[\operatorname{Tail}(u_j;x_0,r_j)]^{p-1}+1\bigg).
\end{align}

Using \eqref{eq5.7} and \eqref{eq5.16}, we get
\begin{align}\label{eq5.17}
    \fint_{\{ 2B_{j+1}\}} |q-(q)_{ 2B_{j+1}}| \d x \leq & C \left(\tau^{1-p}\kappa^{p-\sigma(p-1)}+1\right)\leq C.
\end{align}
Applying \eqref{eq5.14} and \eqref{eq5.17}, we obtain
\begin{align*}
    \frac{|\{u_j \leq 2 \tau\phi(r_j)\}\cap 2B_{r_{j+1}}(x_0)|}{|2B_{j+1}(x_0)|}=\frac{|{\{q=\epsilon\} \cap 2B_{j+1}}|}{|2B_{j+1}|}\leq \frac{C}{\epsilon} \leq \frac{\bar{C}}{\log\frac{1}{\kappa}}.
\end{align*}
\textbf{Step 2:-}
Finally we are going to use an iterative argument to show \eqref{eq5.3} for $i=j+1$. Initially for any $i\in \mathbb{N}\cup \{0\}$, we denote $\gamma_i=\left(1+\frac{1}{2^i}\right)r_{j+1}$, $\bar{\gamma_i}=\frac{\gamma_i+\gamma_{i+1}}{2}$, $B^i=B_{\gamma_i}(x_0)$, $\bar{B^i}=B_{\bar{\gamma_i}}(x_0)$, $\rho_i=\left(1+\frac{1}{2^i}\right)\tau \phi(r_j)$ and $D^i=B^i \cap \{u_j \leq \rho_i\}$. Note that $r_{j+1}\leq \gamma_{i+1}\leq \bar{\gamma_i}\leq \gamma_i \leq 2 r_{j+1}<r_j$. 
\par Let us consider cut-off function $w_i\in C_c^\infty(\bar{B^i})$ such that $0\leq w_i\leq 1$ in $\bar{B^i}$,
     $w_i=1$ in $B^{i+1}$ and $|\nabla w_i|\leq \frac{C2^{i}}{\gamma_i}$ in $\bar{B^i}$ with $C:=C(N,p)>0$. Finally, let $\phi_i=(\rho_i-u_j)_+$. Note that $\rho_i-\rho_{i+1}=\frac{\tau \phi(r_j)}{2^{i+1}}$ and $\phi_i \leq \rho_i\leq 2\tau \phi(r_j)$ in $B^i$. 
     
     The rest of the proof follows {\it verbatim}  from \cite[Lemma $5.1$]{BGKL2026}, with an exception of the \textbf{Estimate of $Q_2$}. We discuss the following modified estimate of $Q_2,$ under the new tail \eqref{Tail}.

\noindent \textbf{New estimate of $Q_2$:} We know from the proof of \cite[Lemma $5.1$]{BGKL2026} that
\begin{align}\label{eq5.23}
    Q_2=\frac{C r^{p}_{j+1} }{|B^i|} \int_{(0,1)} C_{N,s,p} \left( \ess_{x\in \supp \{w_i\}} \int_{\mathbb{R}^N\setminus B^i} \frac{{\phi_i}^{p-1}(y)}{|x-y|^{N+sp}}\d y . \int_{B^i}{\phi_i} w_i^p \d x \right)\d \mu. 
\end{align}
Observe that for $x\in \supp\{ w_i\}=\bar{B^i}$ and $y\in \mathbb{R}^N\setminus B^i$, we get
\begin{align}\label{eq5.24}
    \frac{1}{|y-x|}=\frac{|y-x_0|}{|y-x|} \frac{1}{|y-x_0|}\leq \frac{1}{|y-x_0|} \left( 1+\frac{\bar{\gamma_i}}{\gamma_i-\bar{\gamma_i}}\right)\leq \frac{1}{|y-x_0|}  2^{i+3}
\end{align}
 and
 \begin{align}\label{eq5.25}
     \int_{B^i}{\phi_i} w_i^p \d x\leq C(\tau\phi(r_j))|D^i|,
 \end{align}
 for some constant $C:=C(N,p)>0$. 
Now  using \eqref{eq5.8}, we obtain 
 \begin{align}\label{eq5.26}
     [\operatorname{Tail}(\phi_i;x_0,r_{j+1})]^{p-1}\leq & C \bigg[\int_{(0,1)} C_{N,s,p}\left(\int_{B_j\setminus B_{j+1}} \frac{r^{p}_{j+1}\phi_i^{p-1}}{|x-x_0|^{N+sp}}\d x\right)\d \mu \nonumber \\
     &+  \left( \frac{r_{j+1}}{r_j}\right)^{p} [\operatorname{Tail}(\phi_i;x_0,r_{j})]^{p-1}\bigg] \nonumber \\
      \leq &  C\bigg[\int_{(0,1)}C_{N,s,p}{\left(\int_{B_j\setminus B_{j+1}} \frac{r^{p}_{j+1}(\tau \phi(r_j))^{p-1}}{|x-x_0|^{N+sp}}\d x\right)\d \mu  }\nonumber \\
      &+ \kappa^{p} [\operatorname{Tail}(u_j;x_0,r_{j})]^{p-1}\bigg] \nonumber \\
    \leq &  C \left((\tau \phi(r_j))^{p-1} + \kappa^{p } \kappa^{-\sigma(p-1)} \phi^{p-1}(r_{j}) \right) \nonumber  \\
     \leq &  C \left( 1+\frac{k^{p-\sigma(p-1)}}{\tau^{p-1}} \right)(\tau \phi(r_j))^{p-1} \leq C (\tau \phi(r_j))^{p-1},
 \end{align} 
 where for some constant $C:=C(N,p,\mu, \Sigma)>0$.
Using \eqref{eq5.24}, \eqref{eq5.25} and \eqref{eq5.26} in \eqref{eq5.23}, we get
 \begin{align*}
     Q_2 \leq & \frac{C }{|B^i|}(\tau\phi(r_j))|D^i| \sup_{s\in \Sigma}\{2^{i(N+sp)}\}  [\operatorname{Tail}(\phi_i;x_0,r_{j+1})]^{p-1} \nonumber \\
      \leq &C \frac{|D^i| }{|B^i|} 2^{i(N+p\sup_{s\in \Sigma}s)} (\tau\phi(r_j))^p.
 \end{align*}
This completes the proof.
\end{proof}

 \noindent \textbf{Proof of Theorem \ref{Holder t}}: The result follows from the Lemma \ref{lmn5.1}.

\section{Harnack Inequalities} \label{sec5}

The following lemma establishes the expansion-of-positivity technique, which works well with superposition operators.
\begin{lem}\label{lmn6.1}
    Let $u$ be a weak supersolution of \eqref{MP} with $u\geq 0$ in $B_R(x_0)\subset \Omega$. Suppose that $\zeta\geq 0$ and that there exists $\sigma\in (0,1]$ such that 
    \begin{align*}
        |\{u\geq \zeta\} \cap B_r(x_0)| \geq \sigma |B_r(x_0)|,
    \end{align*}
    for some $r\in( 0,1]$ with $0<16r<R$. Then there exist a constant $C:=C(N,p,\mu,\Sigma)>0$ such that
    \begin{align*}
        |B_{6r}(x_0)\cap \{u\leq 2\delta \zeta-{ \frac{1}{2}\bigg(\frac{r}{R}\bigg)^{\frac{p}{p-1}}  \operatorname{Tail}\big(u_-;x_0,R\big)-\epsilon \} }| \leq \frac{C|B_{6r}(x_0)|}{\sigma \log \frac{1}{2\delta}},
    \end{align*}
    for any $\delta \in (0,\frac{1}{4})$ and $\epsilon>0$.
\end{lem}
\begin{proof}
    Let $w\in C_c^\infty(B_{7r}(x_0))$ such that $w(x)=1$ for all $x\in B_{6r}(x_0)$, $w(x)\in [0,1]$ for all $x\in B_{7r}(x_0)$ with $|\nabla w|\leq \frac{8}{r}$. We set $v=u+d_\epsilon$, where
    \begin{align*}
        d_\epsilon=\frac{1}{2}\bigg(\frac{r}{R}\bigg)^{\frac{p}{p-1}}  \operatorname{Tail}\big(u_-;x_0,R\big)+\epsilon.
    \end{align*}
    Then, we see that $v$ is a weak supersolution of \eqref{MP}. We choose $\psi=v^{1-p}w^p$ as a test function.
    Now, the proof follows the same fashion as in the proof presented in \cite[Lemma $6.1$]{BGKL2026}, except for the \textbf{Estimate of $L'_2$}. We will give the modified estimate of $L_2'$ here.\\
\noindent \textbf{New Estimate of $L_2'$:}
        Using the support of $w$ in $B_{7r}(x_0)$, the definition of $v$ and the fact that  $v_-\leq u_-$, we get
            \begin{align*}
                L_2'=&\int_{(0,1)}C_{N, s, p}\left(\int_{\mathbb{R}^N\setminus {B_{8r}(x_0)} \cap \{v(y)<0\}}\int_{B_{8r}(x_0)} \frac{{|v(x)+v_-(y)|^{p-1}}{(v^{1-p}(x)w^p(x))}}{|x-y|^{N+s p}} \d x \d y\right) \d \mu(s)\nonumber\\
                \leq & C  r^N \int_{(0,1)}C_{N, s, p}\left( \int_{\mathbb{R}^N\setminus {B_{8r}(x_0)}} \bigg(1+\frac{u_-(y)}{d_\epsilon}\bigg)^{p-1} {\frac{8^{N+sp}}{|x_0-y|^{N+s p}}} \right) \d \mu(s) \nonumber\\
                 \leq & C r^N\sup_{s\in \Sigma}8^{N+sp}  \sup_{s\in\Sigma} r^{-sp} + C r^N \sup_{s\in \Sigma}8^{N+sp}  \bigg(\frac{1}{R}\bigg)^{p}  d_\epsilon^{1-p} [\operatorname{Tail}(u_-,x_0,R)]^{p-1} \leq C r^{N-p}.
            \end{align*}

This completes the proof.
\end{proof}

\begin{lem}\label{lmn6.2}
    Let $u$ be a weak supersolution of \eqref{MP} with $u\geq 0$ in $B_R(x_0)\subset \Omega$. Suppose that $\zeta\geq 0$ and that there exists $\sigma\in (0,1]$ such that 
    \begin{align*}
        |\{u\geq \zeta\} \cap B_r(x_0)| \geq \sigma |B_r(x_0)|,
    \end{align*}
    for some $r\in( 0,1]$ with $0<16r<R$. There exists a constant $\delta \in (0,\frac{1}{4})$ such that
    \begin{align}\label{eq6.19}
        \essi_{B_{4r}(x_0)} u \geq \delta \zeta -\bigg(\frac{r}{R}\bigg)^{\frac{p}{p-1}}  \operatorname{Tail}\big(u_-;x_0,R\big).
    \end{align}
\end{lem}
\begin{proof} First, we prove that, for any $\epsilon>0$, there exists a constant $\delta \in \bigg(0, \frac{1}{4}\bigg)$ such that 
\begin{align}\label{eq6.20}
    \essi_{B_{4r}(x_0)} u \geq \delta \zeta -\bigg(\frac{r}{R}\bigg)^{\frac{p}{p-1}}  \operatorname{Tail}\big(u_-;x_0,R\big)-2\epsilon.
\end{align}
Therefore, as a limiting case of \eqref{eq6.20}, the property \eqref{eq6.19} is established.
  Now, we turn our attention to show \eqref{eq6.20}. For this, let us consider 
  \begin{align} \label{eq6.ee}
  \delta \zeta -\bigg(\frac{r}{R}\bigg)^{\frac{p}{p-1}}  \operatorname{Tail}\big(u_-;x_0,R\big) -2\epsilon \geq0.
  \end{align}
  Otherwise, since $u\geq 0$ in $B_R(x_0)$, the inequality $\eqref{eq6.19}$ is trivially true in case \eqref{eq6.ee} does not hold. 

From now on, we proceed to the prove this result as demonstrated in \cite[Lemma $6.2$]{BGKL2026} except the \textbf{Estimate of $B_3$}, which we derive below.

For $j\in \mathbb{N}\cup \{0\}$, we set
    \begin{align*}
        l=\zeta_j=\delta \zeta+2^{-(j+1)} \delta \zeta,~ \quad \gamma=\gamma_j=4r+2^{1-j}r,~ \quad \bar{\gamma_j}=\frac{\gamma_j+\gamma_{j+1}}{2}.
    \end{align*}
    Then $l\in (\delta \zeta, 2\delta \zeta)$, {$\gamma_j \in (4r,6r]$}, $\bar{\gamma_j} \in (4r,6r)$ and 
    \begin{align*}
        \zeta_j-\zeta_{j+1}=2^{-(j+2)}\delta \zeta \geq 2^{-(j+3)}\zeta_j.
    \end{align*}
    We denote $B_j=B_{\gamma_j}(x_0)$, $\bar{B_i}=B_{\bar{\gamma_j}}(x_0)$ and note that 
    \begin{align*}
        v_j=(\zeta_j-u)_+\geq  2^{-(j+3)}\zeta_j \chi_{\{u<\zeta_{j+1}\}},
    \end{align*}
    where $$\chi_A(x)=\begin{cases}
        &1 \text{ if } x\in A,\\
        &0 \text{ if } x\notin A.
    \end{cases}$$ Moreover $\zeta_j$, $\gamma_j$ are monotonically decreasing and $\gamma_{j+1}<\bar{\gamma_j}<\gamma_j$.
    
    Let $w_j \in C_c^\infty(\bar{B_j})$ be a sequence of function such that $w_i=1$ in $B_{j+1}$ and $0\leq w_i\leq 1$ with $|\nabla w_j|\leq \frac{2^{j+3}}{r}$ in $\bar{B_j}$.

 \noindent  \textbf{New Estimate of $B_3$:} Observe that, for any $ x\in \supp \{w_j\}=\bar{B_j}$ and $ y\in \mathbb{R}^N\setminus B_j$, we get
 
 \begin{align}\label{eq6.32}
    \frac{|y-x_0|}{|y-x|}\leq \frac{|y-x|+|x-x_0|}{|y-x|}\leq 1+\frac{\bar{\gamma_j}}{\gamma_j-\bar{\gamma_j}} \leq { 2^{j+5}}.
\end{align}
Using \eqref{eq6.32} along with the convexity we get
    \begin{align*}
        B_3\leq & \zeta_j |B_{\gamma_j}(x_0)\cap \{u<\zeta_j\}| \int_{(0,1)}C_{N, s, p}\left(2^{(j+5)(N+sp)}\int_{\mathbb{R}^N\setminus {B_{\gamma_j}(x_0)} }\frac{{ (\zeta_j+u_-(y))^{p-1}}{}}{|y-x_0|^{N+s p}} \d y \right) \d \mu(s)\nonumber \\
       \leq & \zeta_j |B_{\gamma_j}(x_0)\cap \{u<\zeta_j\}| \int_{(0,1)}C_{N, s, p}\bigg(2^{(j+5)(N+sp)}  \bigg\{ \zeta_j^{p-1} {\frac{r^{-sp}}{sp} }\nonumber \\
       &\quad\quad\quad +r^{-p} \bigg(\frac{r}{R}\bigg)^{p}R^{p} \int_{\mathbb{R}^N\setminus {B_{R}(x_0)}} \frac{u_-(y))^{p-1}}{|y-x_0|^{N+s p}} \bigg\} \d \mu(s) \nonumber \\
       \leq &C\sup_{s\in \Sigma}2^{j(N+sp)} \zeta_j |B_{\gamma_j}(x_0)\cap \{u<\zeta_j\}| r^{-p}\bigg(\zeta_j^{p-1}+\bigg(\frac{r}{R}\bigg)^{p}[\operatorname{Tail}(u_-;x_0,R)]^{p-1}\bigg)\nonumber \\
       \leq & C\sup_{s\in \Sigma}2^{j(N+sp)} \zeta_j^{p} |B_{\gamma_j}(x_0)\cap \{u<\zeta_j\}| r^{-p}.
    \end{align*}
 This completes the proof.
\end{proof}

Next, we present preliminary version of weak Harnack inequality based on the argument as in \cite[Lemma 4.1]{DKP2014} (see \cite[Lemma 6.4]{BGKL2026}). We briefly sketch the proof.

\begin{lem} \label{lmn6.4}

    Let $u$ be a weak supersolution of \eqref{MP} with $u\geq 0$ in $B_R(x_0)\subset \Omega$. There exist positive constants $\kappa:=\kappa(N,p,\mu,\Sigma) \in(0,1)$ and $C:=C(N,p,\mu,\Sigma)\in [1, \infty)$ such that
    \begin{align*}
         \bigg(\fint_{B_{{r}}(x_0)}u^\kappa \d x \bigg)^{\frac{1}{\kappa}} \leq C \essi_{B_r(x_0)}u +C{\bigg(\frac{r}{R}\bigg)^{\frac{p}{p-1}}} \operatorname{Tail}\big(u_-;x_0,R\big),
    \end{align*}
    where $B_r(x_0) \subset B_R(x_0)$ with $r\in (0,1]$.
\end{lem}

\begin{proof}
 For $\epsilon>0$ and $i=0,1,2,3,....$, we set
 \begin{align*}
     E^i_t:=\bigg\{x\in B_r(x_0):u(x)>\epsilon \delta^i -\frac{F}{1-\delta} \bigg\},
 \end{align*}
 where $\delta$ is given in Lemma \ref{lmn6.2} and 
 \begin{align*}
     F:=\bigg(\frac{r}{R}\bigg)^{\frac{p}{p-1}}\operatorname{Tail}\big(u_-;x_0,R\big).
 \end{align*}
 The remaining part runs parallel to  \cite[Lemma 4.1]{DKP2014}, together with Lemma \ref{lmn6.30}.
 \end{proof}

We now establish a Caccioppoli-type estimate for the new superposition tail  \eqref{Tail}.
\begin{lem}\label{lmn7.1}
    Let $p\in(1, \infty)$, $d>0$ and $q\in(1,p)$. Suppose that $u$ is a weak supersolution of \eqref{MP} such that $u\geq 0$ in $B_R(x_0)\subset \Omega$ and $\phi=(u+d)^\frac{p-q}{p}$. Then there exists a constant {$C:=C(p,q,N, \Sigma,\mu)>0$} such that
    \begin{align*}
       \alpha & \int_{B_r(x_0)} w^p|\nabla \phi|^p \d x \leq C\bigg[\alpha \int_{B_r(x_0)} \phi^p|\nabla w|^p \d x \nonumber \\
       &+ \int_{(0,1)} C_{N,s,p} \left( \iint_{B_r(x_0)\times B_r(x_0)} \frac{ \max\{\phi(x), \phi(y)\}^p|w(x)-w(y)|^p}{|x-y|^{N+sp}}\d x\d y\right)\d \mu \bigg] \nonumber \\
       &+C \bigg[ \int_{(0,1)} C_{N,s,p} \left( \ess_{x\in \supp \{w\}} \int_{\mathbb{R}^N\setminus B_r(x_0)}\frac{1}{|x-y|^{N+sp}} \d y \right) \d \mu \nonumber \\
       & + d^{1-p} \bigg(\frac{1}{R}\bigg)^{p}[\operatorname{Tail}\big(u_-(y);x_0,R\big)]^{p-1} \bigg]\bigg(\int_{B_r(x_0)}\phi^p w^p \d x \bigg)
    \end{align*}
    for any $B_r(x_0)\subset B_{\frac{3R}{4}}(x_0)$ and nonnegative function $w\in C_c^\infty(B_r(x_0))$.
\end{lem}
\begin{proof}
    Let $v=u+d$ and $q\in [1+\epsilon,p-\epsilon]$ for $d>0$ and small enough $\epsilon>0$. Then $v$ is a weak supersolution of \eqref{MP}. By selecting $\psi=v^{1-q}w^p$ as a test function, the proof of this result is similar to that of \cite[Lemma $6.5$]{BGKL2026} except the  \textbf{Estimate of $J_2$}.\\
    \noindent \textbf{New estimate of $J_2$:} Observe that \begin{enumerate}
        \item[(i)] $|v(x)-v(y)|^{p-2}(v(x)-v(y))\leq C(v(x))^{p-1}+C(v_-(y))^{p-1}$,
        \item[(ii)] $v^{1-q}(x)\leq d^{1-p}v^{p-q}(x)$,
        \item[(iii)] $(v_-(y))=0$ for all $y\in B_R(x_0)$.
     \end{enumerate}
     Using above observation, we get 
    \begin{align*}
        J_2 \leq& \int_{(0,1)}C_{N, s, p}\left(\int_{\mathbb{R}^N\setminus {B_r(x_0)}}\int_{B_r(x_0)}{C\{(v(x))^{p-1} + v^{p-1}_-(y)\} v^{1-q}(x)w^p(x)}\d \nu \right) \d \mu \nonumber \\
        \leq & C \int_{(0,1)}C_{N, s, p} \bigg( \ess_{x\in \supp \{w\}} \int_{\mathbb{R}^N\setminus B_r(x_0)} \frac{1}{|x-y|^{N+sp}}\d y\nonumber \\
        &+d^{1-p}{4^{N+sp}} \int_{\mathbb{R}^N\setminus B_R(x_0)} (v_-(y))^{p-1} {\frac{1}{|y-x_0|^{N+sp}}} \d y \bigg)\bigg(\int_{B_r(x_0)}\phi^p w^p \d x \bigg) \d \mu \nonumber \\
        \leq & C \int_{(0,1)}C_{N, s, p} \bigg( \ess_{x\in \supp\{ w\}} \int_{\mathbb{R}^N\setminus B_r(x_0)} \frac{1}{|x-y|^{N+sp}}\d y\bigg) \bigg(\int_{B_r(x_0)}\phi^p w^p \d x \bigg) \d \mu \nonumber \\
        &+C d^{1-p} {\sup_{s\in \Sigma} 4^{N+sp} } \bigg(\frac{1}{R}\bigg)^{p}[\operatorname{Tail}\big(v_-(y);x_0,R\big)]^{p-1}  \bigg(\int_{B_r(x_0)}\phi^p w^p \d x \bigg)\nonumber \\
        \leq &  C \bigg\{\int_{(0,1)}C_{N, s, p} \bigg( \ess_{x\in \supp \{w\}} \int_{\mathbb{R}^N\setminus B_r(x_0)} \frac{1}{|x-y|^{N+sp}}\d y \bigg) \d \mu \nonumber \\
        &+ d^{1-p} \bigg(\frac{1}{R}\bigg)^{p}[\operatorname{Tail}\big(v_-(y);x_0,R\big)]^{p-1} \bigg\} \bigg(\int_{B_r(x_0)}\phi^p w^p \d x \bigg). 
    \end{align*}
    
    This completes the proof.
\end{proof}

Subsequently, we demonstrate the tail estimate \eqref{MP} using the method established in \cite[Lemma 4.2]{DKP2014}.
\begin{lem}[Tail Estimate]\label{lmn6.3}
    Let $u$ be a weak solution of \eqref{MP} and $x_0\in \Omega, R>0$ such that $B_R(x_0)\subset \Omega$ with $u \geq 0$ in $B_R(x_0)$. Then, there exists a constant $C:=C(N,p,\Sigma, \mu)>0$ such that, for any $0<r<R$ with $r\in (0,1]$, the following holds:
    $$\operatorname{Tail}(u_+;x_0,r)\leq C \ess_{B_r(x_0)}u +C\bigg( \frac{r}{R}\bigg)^\frac{p}{p-1}\operatorname{Tail}(u_-;x_0,R).$$
\end{lem}
\begin{proof}
    Let $w\in C_c^\infty(B_r(x_0))$ be a function such that $0\leq w\leq 1$, $|\nabla w|\leq \frac{8}{r}$ and $w \equiv 1$ in $B_{\frac{r}{2}}(x_0)$. Define $l:=\ess_{B_r(x_0)}u$ and $v=u-2l$. Since $u$ is a weak solution to \eqref{MP}, taking $\psi:=vw^p$ as a test function, we get
    \begin{align}\label{6.3.1}
        0=&\int_{(0,1)}C_{N, s, p}\left(\iint_{\mathbb{R}^{2 N}} \frac{ \mathcal{A}(u(x,y))(\psi(x)-\psi(y))}{|x-y|^{N+s p}} \d x \d y\right) \d \mu(s)\nonumber \\ 
& + \alpha \int_{\Omega} |\nabla u(x)|^{p-2}\nabla u(x) \cdot \nabla \psi(x) \d x\nonumber\\
=&2\int_{(0,1)}C_{N, s, p} \left(\int_{\mathbb{R}^{N}\setminus B_r(x_0)}\int_{B_r(x_0)} \frac{ \mathcal{A}(u(x,y))\psi(x)}{|x-y|^{N+s p}} \d x \d y\right) \d \mu(s)\nonumber\\
& +\int_{(0,1)}C_{N, s, p}\left(\iint_{B_r(x_0)\times B_r(x_0)} \frac{ \mathcal{A}(u(x,y))(\psi(x)-\psi(y))}{|x-y|^{N+s p}} \d x \d y\right) \d \mu(s)\nonumber\\
& \quad +\alpha \int_{\Omega} |\nabla u(x)|^{p-2}\nabla u(x) \cdot \left(pw^{p-1}v\nabla w+w^p\nabla v\right) \d x\nonumber\\
:=&J_1+J_2+J_3.
    \end{align}
    \textbf{Estimate of $J_1$:} Clearly,
    \begin{align}\label{6.3.4}
        J_1=&2\int_{(0,1)}C_{N, s, p}\left(\int_{\mathbb{R}^{N}\setminus B_r(x_0)\cap \{u(y)\geq l\}}\int_{B_r(x_0)} \frac{ \mathcal{A}(u(x,y))\psi(x)}{|x-y|^{N+s p}} \d x \d y\right) \d \mu(s)\nonumber\\
        &+2\int_{(0,1)}C_{N, s, p}\left(\int_{\mathbb{R}^{N}\setminus B_r(x_0)\cap \{u(y)< l\}}\int_{B_r(x_0)} \frac{ \mathcal{A}(u(x,y))\psi(x)}{|x-y|^{N+s p}} \d x \d y\right) \d \mu(s)\nonumber\\
        \geq& 2l\int_{(0,1)}C_{N, s, p}\left(\int_{\mathbb{R}^{N}\setminus B_r(x_0)}\int_{B_r(x_0)} \frac{(u(y)-l)_+^{p-1}w^p(x)}{|x-y|^{N+s p}} \d x \d y\right) \d \mu(s)\nonumber\\
        & -4l\int_{(0,1)}C_{N, s, p}\left(\int_{\mathbb{R}^{N}\setminus B_r(x_0)\cap \{u(y)< l\}}\int_{B_r(x_0)} \frac{(u(x)-u(y))_+^{p-1}w^p(x)}{|x-y|^{N+s p}} \d x \d y\right) \d \mu(s)\nonumber\\
        :=&2J'_1-4J''_1.
    \end{align}
   \textbf{Estimate of $J'_1$:}
   Observe that $w \equiv 1$ in $B_{\frac{r}{2}}(x_0)$ and $|x-y|\leq 2|x_0-y|$ for all $x\in B_{{r}}(x_0), y\in \mathbb{R}^{N}\setminus B_r(x_0)$. Therefore, we have
    \begin{align}\label{6.3.5}
        J'_1
        \geq &l\int_{(0,1)}C_{N, s, p}\left(\int_{\mathbb{R}^{N}\setminus B_r(x_0)}\int_{B_r(x_0)} \frac{(u(y)-l)_+^{p-1}w^p(x)}{2^{N+sp}|x_0-y|^{N+s p}} \d x \d y\right) \d \mu(s) \nonumber \\
        \geq& Cl\int_{(0,1)} C_{N, s, p} \left(\int_{\mathbb{R}^{N}\setminus B_r(x_0)}\int_{B_{\frac{r}{2}}(x_0)} \frac{u_+(y)^{p-1}w^p(x)}{2^{N+sp}|x_0-y|^{N+s p}} \d x \d y\right) \d \mu(s)\nonumber\\
        &  -Cl^p \int_{(0,1)}C_{N, s, p} \left(\int_{\mathbb{R}^{N}\setminus B_r(x_0)}\int_{B_r(x_0)} \frac{w^p(x)}{2^{N+sp}|x_0-y|^{N+s p}} \d x \d y\right) \d \mu(s)\nonumber\\
       \geq &{ Cl\int_{(0,1)}C_{N, s, p}|B_{\frac{r}{2}}(x_0)|\left(\int_{\mathbb{R}^{N}\setminus B_r(x_0)} \frac{u_+(y)^{p-1}}{2^{N+sp}|x_0-y|^{N+s p}} \d y\right) \d \mu(s)}\nonumber\\
         & -Cl^p\int_{(0,1)}C_{N,s,p}|B_r(x_0)|\frac{1}{s}\frac{1}{2^{N+sp}}r^{-sp}\d \mu(s)\nonumber\\
        \geq & Cl|B_r(x_0)|r^{-p}\left(\inf\limits_{s\in \Sigma}\frac{1}{2^{N+sp}}\right)\left[\operatorname{Tail}(u_+;x_0,r)\right]^{p-1}\nonumber \\
        &-\bigg(\sup_{s\in \Sigma}\frac{1}{2^{N+sp}} \bigg)Cl^p|B_r(x_0)|\left(\sup\limits_{s\in \Sigma\mu}r^{-sp}\right)\nonumber\\
        \geq& Cl|B_r(x_0)|r^{-p}\left[\operatorname{Tail}(u_+;x_0,r)\right]^{p-1}-Cl^p|B_r(x_0)|r^{-p}.
    \end{align}
    \textbf{Estimate of $J''_1$:} Using the inequality
    \begin{equation}\label{6.3.10}
      { (a+b)^q\leq \max\{2^{q-1},1\}(a^q+b^q), a>0,b>0 \text{ and }q>0}
    \end{equation}
    for $q=p-1$ and proceeding similar to \eqref{6.3.5}, we get
    \begin{align}\label{6.3.6} 
       J_1'' &\leq Cl\int_{(0,1)}C_{N, s, p}\left(\int_{\mathbb{R}^{N}\setminus B_R(x_0)}\int_{B_{r}(x_0)} \frac{(l+u_-(y))^{p-1}w^p(x)}{|x-y|^{N+s p}} \d x \d y\right) \d \mu(s)\nonumber\\
        &\quad +Cl\int_{(0,1)}C_{N, s, p}\left(\int_{B_R(x_0)\setminus B_r(x_0)}\int_{B_r(x_0)} \frac{l^{p-1}w^p(x)}{|x-y|^{N+s p}} \d x \d y\right) \d \mu(s)\nonumber\\
        &\leq Cl|B_r|R^{-p}\left[\operatorname{Tail}(u_-;x_0,R)\right]^{p-1}+Cl^p|B_r|r^{-p}.
    \end{align}
    Combining \eqref{6.3.4}--\eqref{6.3.6}, we deduce
    \begin{align}\label{6.3.7}
        J_1&\geq  -Cl|B_r|R^{-p}\left[\operatorname{Tail}(u_-;x_0,R)\right]^{p-1}\nonumber\\
        &\quad -Cl^p|B_r|r^{-p}+Cl|B_r|r^{-p}\left[\operatorname{Tail}(u_+;x_0,r)\right]^{p-1}.
    \end{align}
    \textbf{Estimate of $J_2$:} Estimating similar to $J_1$, we get
    \begin{align}\label{6.3.8}
        J_2&\geq \int_{(0,1)}C_{N, s, p}\left(\iint_{B_r(x_0)\times B_r(x_0)} \frac{ \mathcal{A}(u(x,y))\psi(x)}{|x-y|^{N+s p}} \d x \d y\right) \d \mu(s)\nonumber\\
        &\quad -\int_{(0,1)}C_{N, s, p}\left(\iint_{B_r(x_0)\times B_r(x_0)} \frac{ \mathcal{A}(u(x,y))\psi(y)}{|x-y|^{N+s p}} \d x \d y\right) \d \mu(s)\nonumber\\
        &{=-2 \int_{(0,1)}C_{N, s, p}\left(\iint_{B_r(x_0)\times B_r(x_0)} \frac{ \mathcal{A}(u(x,y))(u(y)-2l)w(y)^p}{|x-y|^{N+s p}} \d x \d y\right) \d \mu(s)}\nonumber\\
        &\geq -Cl^p\int_{(0,1)} C_{N, s, p} \left(\iint_{B_r(x_0)\times B_r(x_0)} \frac{w(y)^p}{|x-y|^{N+s p}} \d x \d y\right) \d \mu(s)\nonumber\\
        &\geq -Cl^p|B_r(x_0)|\left(\sup\limits_{s\in \operatorname{supp}\mu}r^{-sp}\right)\nonumber\\
        &\geq -Cl^p|B_r(x_0)|r^{-p}.
    \end{align}
    \textbf{Estimate of $J_3$:} Recall that we have $|\nabla w|\leq \frac{8}{r}$ in $B_r(x_0)$. Also, $|v|=|u-2l|\leq |u|+2l\leq 3l$ in $B_r(x_0)$ and $\nabla (u-2l)=\nabla u$. Thus, using Young's inequality, we deduce
    \begin{align}\label{6.3.2}
        pw^{p-1}v|\nabla u|^{p-2}\nabla u.\nabla w&\leq \frac{1}{2}w^p|\nabla u|^p+C(p)|u-2l|^p|\nabla w|^p \nonumber\\
        &\leq \frac{1}{2}w^p|\nabla u|^p+C(p)l^p|\nabla w|^p.
    \end{align}
    Substituting \eqref{6.3.2} in $J_3$, we get
    \begin{align}\label{6.3.3}
        J_3&\geq \int_{B_r(x_0)}\left(\frac{1}{2}w^p|\nabla u|^p-C(p)l^p|\nabla w|^p\right)\d x \nonumber\\
        &\geq -C(p)\int_{B_r(x_0)}l^p|\nabla w|^p \d x \nonumber\\
        &\geq -C(p)l^pr^{-p}|B_r(x_0)|.
    \end{align}
    Substituting the estimates obtained in \eqref{6.3.7}, \eqref{6.3.8}, and \eqref{6.3.3}, in \eqref{6.3.1}, and dividing throughout by $r^{-p}l|B_r|$, we deduce
    \begin{align}\label{6.3.9}
       \left[\operatorname{Tail}(u_+;x_0,r)\right]^{p-1}&\leq Cl^{p-1}+C\bigg(\frac{r}{R}\bigg)^p\left[\operatorname{Tail}(u_-;x_0,R)\right]^{p-1}.
    \end{align}
    Taking $\frac{1}{(p-1)}$-th power in both sides of \eqref{6.3.9} and applying \eqref{6.3.10} with $q=\frac{1}{p-1}$, we obtained the desired result.
\end{proof}

We now prove the Harnack inequality:

\subsection{Proof of the Theorem \ref{T1.5}}
\begin{proof}
   From Lemma \ref{nw}-$(c)$, we get $u$ is a weak subsolution to \eqref{MP}. Consider $0<r'<r$. Then, by Theorem \ref{bddness}, there exists a constant $C:=C(N,p,\Sigma,\mu)>0$ such that 
   \begin{equation}\label{8.3-1}
        \ess_{B_{\frac{r'}{2}(x_0)}} u \leq
            \delta \operatorname{Tail}\bigg(u_+;x_0,\frac{r'}{2}\bigg)+C\delta^{-\frac{(p-1)\eta}{(\eta-1)p}}\bigg( \fint_{B_{r'}(x_0)}u^p \d x \bigg)^\frac{1}{p}.
    \end{equation}
    Now, \eqref{8.3-1} and Lemma \ref{lmn6.3} gives
    \begin{align}\label{8.3-2}
        \ess_{B_{\frac{r'}{2}(x_0)}} u &\leq
            \delta \Big(C\ess_{B_{r'}(x_0)}u +C\bigg( \frac{r'}{R}\bigg)^\frac{p}{p-1}\operatorname{Tail}(u_-;x_0,R)\Big)\nonumber\\
            &\quad +C\delta^{-\frac{(p-1)\eta}{(\eta-1)p}}\bigg( \fint_{B_{r'}(x_0)}u^p \d x \bigg)^\frac{1}{p},
    \end{align}
    where $C=C(N,p,\Sigma,\mu)>0$. Now, consider $\epsilon_1, \epsilon_2 \in [\frac{1}{2},1]$ such that $\epsilon_1<\epsilon_2$ and take $r'=\epsilon_2 r-\epsilon_1 r$. Then, for any $y\in B_{\epsilon_1 r} (x_0)$, clearly $B_{r'} (y) \subset B_{\epsilon_2 r} (x_0)$. From \eqref{8.3-2}, we have
    \begin{align}\label{8.3-3}
        \ess_{B_{\frac{r'}{2}}(y)} u &\leq
            C\delta \Big(\ess_{B_{r'}(y)}u +C\bigg( \frac{r'}{R}\bigg)^\frac{p}{p-1}\operatorname{Tail}(u_-;x_0,R)\Big)\nonumber\\
            &\quad \quad \quad +C\delta^{-\frac{(p-1)\eta}{(\eta-1)p}}\bigg( \fint_{B_{r'}(y)}u^p \d x \bigg)^\frac{1}{p} \nonumber\\
            &\leq C \delta \Big(\ess_{B_{\epsilon_2 r}(x_0)}u +\bigg( \frac{r}{R}\bigg)^\frac{p}{p-1}\operatorname{Tail}(u_-;x_0,R)\Big) \nonumber\\
            & \quad \quad \quad+C\delta^{-\frac{(p-1)\eta}{(\eta-1)p}}(\epsilon_2-\epsilon_1)^{-\frac{N}{p}}\bigg( \fint_{B_{\epsilon_2 r}(x_0)}u^p \d x \bigg)^\frac{1}{p}.
    \end{align}
    Clearly, the set $\{B_\frac{r'}{2}(y): y\in B_{\epsilon_1 r}(x_0)\}$ covers $B_{\epsilon_1 r}(x_0)$. Let $t\in (0,p)$. Using \eqref{8.3-3} and Young's inequality, we get
    \begin{align}\label{8.3-4}
        \ess_{B_{\epsilon_1 r}(x_0)} u &\leq C \delta \Big(\ess_{B_{\epsilon_2 r}(x_0)}u +\bigg( \frac{r}{R}\bigg)^\frac{p}{p-1}\operatorname{Tail}(u_-;x_0,R)\Big) \nonumber\\
        & \quad +C\delta^{-\frac{(p-1)\eta}{(\eta-1)p}}(\epsilon_2-\epsilon_1)^{-\frac{N}{p}}(\ess_{B_{\epsilon_2 r}(y)} u)^{\frac{p-t}{p}}\bigg( \fint_{B_{\epsilon_2 r}(y)}u^t \d x \bigg)^\frac{1}{p} \nonumber\\
        &\leq C\delta \Big(\ess_{B_{\epsilon_2 r}(x_0)}u +\bigg( \frac{r}{R}\bigg)^\frac{p}{p-1}\operatorname{Tail}(u_-;x_0,R)\Big) \nonumber\\
        & \quad +C_1 \delta^{-\frac{(p-1)\eta}{t(\eta-1)}}(\epsilon_2-\epsilon_1)^{-\frac{N}{t}}\bigg( \fint_{B_{ r}(x_0)}u^t \d x \bigg)^\frac{1}{t},
    \end{align}
    with $C_1=C(N,p,\Sigma,\mu,\delta)>0$. Choosing $\delta>0$ sufficiently small, and using \eqref{8.3-4} we have
    \begin{align}\label{8.3-5}
        \ess_{B_{\epsilon_1 r}(x_0)} u&\leq \frac{1}{2}\ess_{B_{\epsilon_2 r}(x_0)}u+C\bigg( \frac{r}{R}\bigg)^\frac{p}{p-1}\operatorname{Tail}(u_-;x_0,R) \nonumber\\
        & \quad + C(\epsilon_2-\epsilon_1)^{-\frac{N}{t}}\bigg( \fint_{B_{ r}(x_0)}u^t \d x \bigg)^\frac{1}{t},
    \end{align}
    where the constant $C=C(N,p,\Sigma,\mu,t)>0$. Applying  \cite[Lemma $8.2$]{GK2022} in \eqref{8.3-5} and then using Lemma \ref{lmn6.4} with $t=\kappa$, we deduce
    \begin{align*}
        \ess_{B_{\frac{r}{2}}(x_0)} u& \leq C\bigg( \fint_{B_{ r}(x_0)}u^t \d x \bigg)^\frac{1}{t}+C\bigg( \frac{r}{R}\bigg)^\frac{p}{p-1}\operatorname{Tail}(u_-;x_0,R) \\
        & \leq C \essi_{B_r(x_0)}u+C\bigg( \frac{r}{R}\bigg)^\frac{p}{p-1}\operatorname{Tail}(u_-;x_0,R),
    \end{align*}
    where $C=C(N,p,\Sigma,\mu)$. This completes the proof.
\end{proof}

Now, we are ready to establish the weak Harnack inequality.

\subsection{Proof of the Theorem \ref{T1.6}}
    \begin{proof}
        We show the result only for $1<p<N$. The conclusion holds when $p\geq N$ following a similar argument. Let $0<r\leq 1$, $\frac{1}{2}<m'<m\leq \frac{3}{4}$ and let $w\in C_c^\infty(B_{mr}(x_0))$ be such that $w(x)=1$ for all $x\in B_{m'r}(x_0)$, $w(x)\in[0,1]$ for all $x\in B_{mr}(x_0)$ with $|\nabla w|\leq \frac{4}{(m-m')r}$. Let $q\in (1,p)$ and $d>0$, we set $v=u+d$ and $\phi=v^{\frac{p-q}{p}}$. Subsequently, we proceed to the proof as discussed in \cite[Theorem $6.6$]{BGKL2026} except the \textbf{Estimate of $F_3$}.
        Note that \begin{align}\label{eq7.10}
             & {\int_{(0,1)} C_{N,s,p} \left( \ess_{x\in \supp \{w\}} \int_{\mathbb{R}^N\setminus B_r(x_0)}\frac{1}{|x-y|^{N+sp}} \d y\right) \d \mu }\nonumber \\
            \leq & C \sup_{s\in \Sigma}r^{-sp} \mu\{(0,1)\} \leq C r^{-p}.
        \end{align}
       Let $\operatorname{Tail}\big(u_-(y);x_0,R\big)$  be strictly positive. For sufficiently small $\epsilon>0$, we have \begin{align}\label{eq7.11}
           d=\frac{1}{2}\bigg(\frac{r}{R}\bigg)^{\frac{p}{p-1}}  \operatorname{Tail}\big(u_-(y);x_0,R\big)+\epsilon >0.
        \end{align}
       \textbf{New Estimate of $F_3$:} Using \eqref{eq7.10} and \eqref{eq7.11}, we obtain
        \begin{align}\label{eq7.12}
            P_3=&\bigg[ \int_{(0,1)} C_{N,s,p} \left( \ess_{x\in \supp \{w\}} \int_{\mathbb{R}^N\setminus B_r(x_0)}\frac{1}{|x-y|^{N+sp}} \right) \d \mu \nonumber \\
       & + d^{1-p}\bigg(\frac{1}{R}\bigg)^{p}[\operatorname{Tail}\big(u_-(y);x_0,R\big)]^{p-1} \bigg]\bigg(\int_{B_r(x_0)}\phi^p w^p \d x \bigg) \nonumber \\
       \leq & \frac{Cr^{-p}}{(m-m')^p} \bigg( \int_{B_{mr}(x_0)} \phi^p \d x\bigg).
        \end{align}
        Again if $\operatorname{Tail}\big(u_-(y);x_0,R\big)$ is zero, then $d=\epsilon>0$. Using \eqref{eq7.10}, we get the same estimate \eqref{eq7.12}.  
       This completes the proof.
    \end{proof}

\section{Conclusion and Remarks}\label{sec6}

\begin{enumerate}
    \item[1.] The local H\"older continuity of the superposition operator in \eqref{MP}, involving the tails ``$\operatorname{Tail}_1$'' in \eqref{tail1} or ``$\operatorname{Tail}_2$'' in \eqref{tail2}, remains open. More precisely, neither a proof nor a counterexample is currently available. The main difficulty arises from the lack of control of the singular term $\frac{1}{s}$ as $s\to0^+$, particularly in the case $\bar{s}:=\inf_{s\in\Sigma}s=0.$
    We refer to \cite{BGKL2026} for further discussion.
 \item[2.] In the case $\bar{s}:=\inf_{s\in\Sigma}s>0,$ the local H\"older continuity of the superposition operator was established in \cite{BGKL2026}. However, the validity of the Harnack inequality remains unresolved when the corresponding tail is given by ``$\operatorname{Tail}_1$'' in \eqref{tail1} or ``$\operatorname{Tail}_2$'' in \eqref{tail2}. The principal obstruction lies in deriving suitable tail estimates for the superposition operator, in particular in controlling the scaling factor $r^{sp}$ appearing in the tail terms.

 \item[3.] The nonlocal superposition tail defined in \eqref{Tail} is particularly well suited for establishing both local H\"older continuity and the Harnack inequality. This is due to the following structural features:
\begin{enumerate}
    \item[(i)] The normalizing constant $C_{N,p,s}$ compensates for the singular behavior of $\frac{1}{s}$ in the limit $s\to0^+$.
    
    \item[(ii)] The factor $r^p$ ensures the appropriate scaling needed to treat the regime $\bar{s}:=\inf_{s\in\Sigma}s>0.$
    
    \item[(iii)] The same factor $r^p$ also plays a crucial role in obtaining tail estimates for the superposition operator, which are fundamental in the analysis of regularity and Harnack-type inequalities.
\end{enumerate}
\end{enumerate}

\section*{Conflict of interest statement}
On behalf of all authors, the corresponding author states that there is no conflict of interest.

\section*{Data availability statement}
Data sharing is not applicable to this article as no datasets were generated or analysed during the current study.

\section*{Acknowledgement}

SB would like to thank the Council of Scientific and Industrial Research (CSIR), India, for financial assistance to carry out this research work [grant no. 09/0874(17164)/ 2023-EMR-I]. SG gratefully acknowledges the financial support for this research work under ARG-MATRICS, grant No: ANRF/ARGM/2025/001570/MTR, Anusandhan National Research Foundation (ANRF), Government of India. This work was completed while SG and VK were visiting the Ghent Analysis \& PDE Center, Ghent University. They gratefully acknowledge the financial support and excellent research facilities provided by the center. RL expresses sincere gratitude for the financial support provided by the Ministry of Education, Government of India.



\end{document}